\documentclass[a4paper,12pt, reqno]{amsart}
\usepackage{latexsym}
\usepackage{color}
\usepackage{amssymb,amsfonts,amsmath,mathrsfs}
\newtheorem{theorem}{Theorem}[section]

\newtheorem{lemma}{Lemma}[section]

\newtheorem{conjecture}{Conjecture}[section]
\newtheorem{remark}{Remark}[section]

\newcommand{\ord}{\text{ord}}

\begin{document}
\title[Hybrid Euler-Hadamard product]{Hybrid Euler-Hadamard product for quadratic Dirichlet L-functions in function fields}
\author{H. M. Bui and Alexandra Florea}
\address{School of Mathematics, University of Manchester, Manchester M13 9PL, UK}
\email{hung.bui@manchester.ac.uk}
\address{Department of Mathematics, Stanford University, Stanford CA 94305, USA}
\email{amusat@stanford.edu}

\allowdisplaybreaks

\begin{abstract}
We develop a hybrid Euler-Hadamard product model for quadratic Dirichlet $L$--functions over function fields (following the model introduced by Gonek, Hughes and Keating for the Riemann-zeta function). After computing the first three twisted moments in this family of $L$--functions, we provide further evidence for the conjectural asymptotic formulas for the moments of the family.
%We compute the first three twisted moments of quadratic Dirichlet $L$-function over function fields. Based on these results we provide further evidence for the conjectural asymptotic formulas for the moments of this family of $L$-functions following the hybrid Euler-Hadamard product model proposed by Gonek, Hughes and Keating [`A hybrid Euler-Hadamard product for the Riemann zeta function', {\it Duke Math. J.} 136 (2007), 507--549].
\end{abstract}
\maketitle

\section{Introduction}

An important and fascinating theme in number theory is the study of moments of the Riemann zeta-function and families of $L$-function. In this paper, we consider the moments of quadratic Dirichlet $L$-functions in the function field setting. Denote by $\mathcal{H}_{2g+1}$ the space of monic, square-free polynomials of degree $2g+1$ over $\mathbb{F}_q[x]$. We are interested in the asymptotic formula for the $k$-th moment,
\[
I_k(g)=\frac{1}{|\mathcal{H}_{2g+1}|}\sum_{D\in\mathcal{H}_{2g+1}}L(\tfrac{1}{2},\chi_D)^k,
\]
as $g\rightarrow\infty$.

In the corresponding problem over number fields, the first and second moments have been evaluated by Jutila [\textbf{\ref{J}}], with subsequent improvements on the error terms by Goldfeld and Hoffstein [\textbf{\ref{GH}}], Soundararajan [\textbf{\ref{S}}] and Young [\textbf{\ref{Y}}], and the third moment has been computed by Soundararajan [\textbf{\ref{S}}]. Conjectural asymptotic formulas for higher moments have also been given, being based on either random matrix theory [\textbf{\ref{KS}}] or the ``recipe'' [\textbf{\ref{CFKRS}}].

Using the idea of Jutila [\textbf{\ref{J}}], Andrade and Keating [\textbf{\ref{AK}}] obtained the asymptotic formula for $I_1(g)$ when $q$ is fixed and $q\equiv1(\textrm{mod}\ 4)$. They explicitly computed the main term, which is of size $g$, and bounded the error term by $O\big(q^{(-1/4+\log_q2)(2g+1)}\big)$. This result was recently improved by Florea [\textbf{\ref{F1}}] with a secondary main term and an error term of size $O_\varepsilon\big(q^{-3g/2+\varepsilon g}\big)$. Florea's approach is similar to Young's [\textbf{\ref{Y}}], but in the function field setting, it is striking that one can surpass the square-root cancellation. Florea [\textbf{\ref{F23}}, \textbf{\ref{F4}}] later also provided the asymptotic formulas for $I_k(g)$ when $k=2,3,4$.

For other values of $k$, by extending the Ratios Conjecture to the function field setting, Andrade and Keating [\textbf{\ref{AK2}}] proposed a general formula for the integral moments of quadratic Dirichlet $L$-functions over function fields. Concerning the leading terms, their conjecture reads

\begin{conjecture}
For any $k\in\mathbb{N}$ we have
\[
I_k(g)\sim 2^{-k/2}\mathcal{A}_k\frac{G(k+1)\sqrt{\Gamma(k+1)}}{\sqrt{G(2k+1)\Gamma(2k+1)}}(2g)^{k(k+1)/2}
\]
as $g\rightarrow\infty$, where 
\[
\mathcal{A}_k=\prod_{P\in\mathcal{P}}\bigg[\bigg(1-\frac{1}{|P|}\bigg)^{k(k+1)/2}\bigg(1+\bigg(1+\frac{1}{|P|}\bigg)^{-1}\sum_{j=1}^{\infty}\frac{\tau_{k}(P^{2j})}{|P|^j}\bigg)\bigg]
\]
with $\tau_k(f)$ being the $k$-th divisor function, and $G(k)$ is the Barnes $G$-function.
\label{ak}
\end{conjecture}

\begin{remark}\emph{An equivalent form of $\mathcal{A}_k$ is}
\begin{equation*}
\mathcal{A}_{k}=\prod_{P\in\mathcal{P}}\bigg(1-\frac{1}{|P|}\bigg)^{k(k+1)/2}\bigg(1+\frac{1}{|P|}\bigg)^{-1}\bigg(\frac12\bigg(1-\frac{1}{|P|^{1/2}}\bigg)^{-k}+\frac12\bigg(1+\frac{1}{|P|^{1/2}}\bigg)^{-k}+\frac{1}{|P|}\bigg).
\end{equation*}
\end{remark}

Beside random matrix theory and the recipe, another method to predict asymptotic formulas for moments comes from the hybrid Euler-Hadamard product for the Riemann zeta-function developed by Gonek, Hughes and Keating [\textbf{\ref{GHK}}]. Using a smoothed form of the explicit
formula of Bombieri and Hejhal [\textbf{\ref{BH}}], the value of the Riemann zeta-function at a height $t$ on the critical line can be approximated as a partial Euler product multiplied by a partial Hadamard product over the nontrivial zeros close to $1/2+it$. The partial Hadamard product is expected to be modelled by the characteristic polynomial of a large random unitary matrix as it involves only local information about the zeros. Calculating the moments of the
partial Euler product rigorously and making an assumption (which can be proved in certain cases) about the independence of the two products, Gonek, Hughes and Keating then reproduced the conjecture for the moments of the Riemann zeta-function first put forward by Keating and Snaith [\textbf{\ref{KS2}}]. The hybrid Euler-Hadamard product model has been extended to various cases [\textbf{\ref{BK}}, \textbf{\ref{BK2}}, \textbf{\ref{D}}, \textbf{\ref{H}}, \textbf{\ref{BGM}}].

In this paper, we give further support for Conjecture 1.1 using the idea of Gonek, Hughes and Keating. Along the way, we also derive the first three twisted moments of quadratic Dirichlet $L$-functions over function fields.

\section{Statements of results}

Throughout the paper we assume $q$ is fixed and $q\equiv 1(\textrm{mod}\ 4)$. All theorems still hold for all $q$ odd by using the modified auxiliary lemmas in function fields as in [\textbf{\ref{BF}}], but we shall keep the assumption for simplicity. Let $\mathcal{M}$ be the set of monic polynomials in $\mathbb{F}_q[x]$, $\mathcal{M}_n$ and $\mathcal{M}_{\leq n}$ be the sets of those of degree $n$ and degree at most $n$, respectively. The letter $P$ will always denote a monic irreducible polynomial over $\mathbb{F}_q[x]$.  The set of monic irreducible polynomials is denoted by $\mathcal{P}$. For a polynomial $f\in \mathbb{F}_q[x]$, we denote its degree by $d(f)$, its norm $|f|$ is defined to be $q^{d(f)}$, and the von Mangoldt function is defined by
$$ \Lambda(f) = 
\begin{cases}
d(P) & \mbox{ if } f=cP^j \text{ for some }c \in \mathbb{F}_q^{\times}\ \text{and}\ j\geq 1, \\
0 & \mbox{ otherwise. }
\end{cases}
$$

Note that
$$|\mathcal{H}_d| = 
\begin{cases}
q & \mbox{ if } d=1, \\
q^{d-1}(q-1) & \mbox{ if } d \geq 2.
\end{cases}
$$
For any function $F$ on $\mathcal{H}_{2g+1}$, the expected value of $F$ is defined by
$$ \Big\langle F \Big\rangle_{\mathcal{H}_{2g+1}} := \frac{1}{| \mathcal{H}_{2g+1} |} \sum_{D \in \mathcal{H}_{2g+1}} F(D).$$

The Euler-Hadamard product we use, which is proved in Section 4, takes the following form.

\begin{theorem}\label{HEH}
Let $u(x)$ be a real, non-negative, $C^\infty$-function with mass $1$ and compactly supported on $[q,q^{1+1/X}]$. Let
\begin{equation*}
U(z)=\int_{0}^{\infty}u(x)E_{1}(z\log x)dx,
\end{equation*}
where $E_{1}(z)$ is the exponential integral, $E_{1}(z)=\int_{z}^{\infty}e^{-x}/xdx$. Then for $\emph{Re}(s)\geq0$ we have
\begin{equation*}
L(s,\chi_D)=P_{X}(s,\chi_D)Z_{X}(s,\chi_D),
\end{equation*}
where
\begin{equation*}
P_{X}(s,\chi_D)=\exp\bigg( \sum_{\substack{f\in\mathcal{M}\\d(f)\leq X}}\frac{\Lambda(f)\chi_D(f)}{|f|^{s}d(f)}\bigg)
\end{equation*} and
\begin{equation*}
Z_{X}(s,\chi_D)=\exp\Big(-\sum_{\rho}U\big((s-\rho)\ X\big)\Big),
\end{equation*}
where the sum is over all the zeros $\rho$ of $L(s,\chi_D)$.
\end{theorem}

As remarked in [\textbf{\ref{GHK}}], $P_X(s,\chi_D)$ can be thought of as the Euler product for $L(s,\chi_D)$ truncated to include polynomials of degree $\leq X$, and $Z_X(s,\chi_D)$ can be thought of as the Hadamard product for $L(s,\chi_D)$ truncated to include zeros within a distance $\lesssim 1/X$ from the point $s$. The parameter $X$ thus controls the relative contributions of the Euler and Hadamard products. Note that a similar hybrid product formula was developed independently by Andrade, Keating, Gonek in [\textbf{\ref{AKG}}].

In Section 5 we evaluate the moments of $P_X(\chi_D):=P_X(1/2,\chi_D)$ rigorously and prove the following theorem.

\begin{theorem}\label{theoremP}
Let $0<c<2$. Suppose that $X\leq (2-c)\log g/\log q$. Then for any $k\in\mathbb{R}$ we have
\[
\Big\langle P_X(\chi_D)^k \Big\rangle_{\mathcal{H}_{2g+1}}=2^{-k/2}\mathcal{A}_k\big(e^\gamma X\big)^{k(k+1)/2}+O\big( X^{k(k+1)/2-1}\big).
\]
\end{theorem}

For the partial Hadamard product, $Z_X(\chi_D):=Z_X(1/2,\chi_D)$, we conjecture that

\begin{conjecture}\label{conjectureZ}
Let $0<c<2$. Suppose that $X\leq (2-c)\log g/\log q$ and $X,g\rightarrow\infty$. Then for any $k\geq0$ we have
\[
\Big\langle Z_X(\chi_D)^k \Big\rangle_{\mathcal{H}_{2g+1}}\sim \frac{G(k+1)\sqrt{\Gamma(k+1)}}{\sqrt{G(2k+1)\Gamma(2k+1)}}\Big(\frac{2g}{e^\gamma X}\Big)^{k(k+1)/2}.
\]
\end{conjecture}

In Section 7 we shall provide some support for Conjecture \ref{conjectureZ} using the random matrix theory model as follows. The zeros of quadratic Dirichlet $L$-functions are believed to have the same statistical distribution as the eigenangles $\theta_n$ of $2N\times 2N$ random symplectic unitary matrices with respect to the Haar measure for some $N$. Equating the density of the zeros and the density of the eigenangles suggests that $N=g$. Hence the $k$-th moment of $Z_X(\chi_D)$ is expected to be asymptotically the same as $Z_X(\chi_D)^k$ when the zeros $\rho$ are replaced by the eigenangles $\theta_n$ and averaged over all $2g\times 2g$ symplectic unitary matrices. This random matrix calculation is carried out in Section 7.

We also manage to verify Conjecture \ref{conjectureZ} in the cases $k=1,2,3$. As, from Theorem \ref{HEH}, $Z_X(\chi_D)=L(\tfrac{1}{2},\chi_D)P_X(\chi_D)^{-1}$, that is the same as to establish the following theorem.

\begin{theorem}\label{k123}
Let $0<c<2$. Suppose that $X\leq (2-c)\log g/\log q$. Then we have
\begin{align*}
&\Big\langle L(\tfrac{1}{2},\chi_D)P_X(\chi_D)^{-1} \Big\rangle_{\mathcal{H}_{2g+1}}= \frac{1}{\sqrt{2}}\frac{2g}{e^\gamma X} + O\big(gX^{-2}\big),\\
&\Big\langle L(\tfrac{1}{2},\chi_D)^2P_X(\chi_D)^{-2} \Big\rangle_{\mathcal{H}_{2g+1}}= \frac{1}{12}\Big(\frac{2g}{e^\gamma X}\Big)^3+O\big(g^3X^{-4}\big)
\end{align*}
and
\[
\Big\langle L(\tfrac{1}{2},\chi_D)^3P_X(\chi_D)^{-3} \Big\rangle_{\mathcal{H}_{2g+1}}= \frac{1}{720\sqrt{2}}\Big(\frac{2g}{e^\gamma X}\Big)^6+O\big(g^6X^{-7}\big).
\]
\end{theorem}

Our Theorem \ref{theoremP} and Theorem \ref{k123} suggest that at least when $X$ is not too large relative to $q^g$, the $k$-th moment of $L(1/2,\chi_D)$ is asymptotic to the product of the moments of $P_X(\chi_D)$ and $Z_X(\chi_D)$ for $k=1,2,3$. We believe that this is true in general and we make the following conjecture.

\begin{conjecture}[Splitting Conjecture]
Let $0<c<2$. Suppose that $X\leq (2-c)\log g/\log q$ and $X,g\rightarrow\infty$. Then for any $k\geq0$ we have
\[
\Big\langle L(\tfrac12,\chi_D)^k \Big\rangle_{\mathcal{H}_{2g+1}}\sim \Big\langle P_X(\chi_D)^k \Big\rangle_{\mathcal{H}_{2g+1}}\Big\langle Z_X(\chi_D)^k \Big\rangle_{\mathcal{H}_{2g+1}}.
\]
\end{conjecture}

Theorem \ref{theoremP}, Conjecture \ref{conjectureZ} and the Splitting Conjecture imply Conjecture 1.1.

To prove Theorem \ref{k123} requires knowledge and understanding about twisted moments of quadratic Dirichlet $L$-functions over function fields,
\[
I_k(\ell;g)=\Big\langle L(\tfrac12,\chi_D)^k\chi_D(\ell) \Big\rangle_{\mathcal{H}_{2g+1}}.
\] 
For that we shall compute the first three twisted moments in Section 6 and show that the following theorems hold.

\begin{theorem}[Twisted first moment]\label{tfm}
Let $\ell=\ell_1\ell_2^2$ with $\ell_1$ square-free. Then we have
\begin{align*}
I_1(\ell;g)=&\,\frac{\eta_1(\ell;1)}{|\ell_1|^{1/2}}\bigg(g-d(\ell_1)+1-\frac{\partial_u\eta_1}{\eta_1}(\ell;1)\bigg)+|\ell_1|^{1/6} q^{-4g/3}  P\big(g+d(\ell_1)\big)\\
&\qquad\qquad+O_\varepsilon\big(|\ell|^{1/2}q^{-3g/2+ \varepsilon g}\big),
\end{align*}
where the function $\eta_1(\ell,u)$ is defined in \eqref{eta} and $P(x)$ is a linear polynomial whose coefficients can be written down explicitly.
\end{theorem}

\begin{theorem}[Twisted second moment]\label{tsm}
Let $\ell=\ell_1\ell_2^2$ with $\ell_1$ square-free. Then we have
\begin{align*}
I_2(\ell;g)=&\,\frac{\eta_2(\ell;1)}{24|\ell_1|^{1/2}}\bigg(\sum_{j=0}^{3}\frac{\partial_u^j\eta_2}{\eta_2}(\ell;1)P_{2,j}\big(2g-d(\ell_1)\big)- 6g\sum_{j=0}^{2}\frac{\partial_u^j\kappa_2}{\kappa_2}(\ell;1,1)Q_{2,j}\big(d(\ell_1)\big)\\
&\qquad\qquad\qquad+2\sum_{i=0}^{1}\sum_{j=0}^{3-i}\frac{\partial_u^j\partial_w^i\kappa_2}{\kappa_2}(\ell;1,1)R_{2,i,j}\big(d(\ell_1)\big)\bigg)+O_\varepsilon\big(|\ell|^{1/2}q^{-g+ \varepsilon g}\big),
\end{align*}
where the functions $\eta_2(\ell,u)$ and $\kappa_2(\ell;u,v)$ are defined in \eqref{eta} and \eqref{kappa2}. Here $P_{2,j}(x)$'s are some explicit polynomials of degrees $3-j$ for all $0\leq j\leq 3$. Also, $Q_{2,j}(x)$'s and $R_{2,i,j}(x)$'s are some explicit polynomials of degrees $2-j$ and $3-i-j$, respectively.

As for the leading term we have
\begin{align*}
I_2(\ell;g)=&\,\frac{\eta_2(\ell;1)}{24|\ell_1|^{1/2}}\Big(8g^3-12g^2d(\ell_1)+d(\ell_1)^3\Big)+O_\varepsilon\big(g^2d(\ell)^\varepsilon\big)+O_\varepsilon\big(|\ell|^{1/2}q^{-g+ \varepsilon g}\big).
\end{align*}
\end{theorem}

\begin{theorem}[Twisted third moment]\label{ttm}
Let $\ell=\ell_1\ell_2^2$ with $\ell_1$ square-free. Then we have
\begin{align*}
I_3(\ell;g)=&\,\frac{\eta_3(\ell;1)}{2^56!|\ell_1|^{1/2}}\sum_{j=0}^{6}\frac{\partial_u^j\eta_3}{\eta_3}(\ell;1)P_{3,j}\big(3g-d(\ell_1)\big)\\
&\qquad+\frac{\kappa_3(\ell;1,1)q^4}{(q-1) |\ell_1|^{1/2}}\sum_{N=3g-1}^{3g}\sum_{i_1=0}^{2}\sum_{i_2=0}^{2-i_1}\sum_{j=0}^{6-i_1-i_2}\frac{\partial_u^j\partial_w^{i_1}\kappa_3}{\kappa_3}(\ell;1,1)R_{3,i_1,i_2,j}(\mathfrak{a},g+d)N^{i_2}\\
&\qquad\qquad+O_\varepsilon(|\ell_1|^{-3/4}q^{-g/4+\varepsilon g})+ O_\varepsilon\big(|\ell|^{1/2}q^{-g/2+\varepsilon g}\big),
\end{align*}
where the functions $\eta_3(\ell,u)$ and $\kappa_3(\ell;u,v)$ are defined in \eqref{eta} and \eqref{kappa3}. Here $P_{3,j}(x)$'s are some explicit polynomials of degrees $6-j$ for all $0\leq j\leq 6$. Also, $\mathfrak{a} \in \{0,1\}$ according to whether $N-d(\ell)$ is even or odd, and $R_{3,i_1,i_2,j}(\mathfrak{a},x)$ are some explicit polynomials in $x$ with degree $6-i_1-i_2-j$.

As for the leading term we have
\begin{align*}
I_3(\ell;g)=\frac{\eta_3(\ell;1)}{2^56! |\ell_1|^{\frac{1}{2}}}&\Big(\big(3g-d(\ell_1)\big)^6-73\big(g+d(\ell_1)\big)^6+396g\big(g+d(\ell_1)\big)^5\\
&\qquad\qquad-540g^2\big(g+d(\ell_1)\big)^4\Big)+O_\varepsilon\big(g^5d(\ell)^\varepsilon\big)+ O_\varepsilon\big(|\ell|^{1/2}q^{-g/4+\varepsilon g}\big).
\end{align*}
\end{theorem}

\section{Background in function fields}

We first give some background information on $L$-functions over function fields and their connection to zeta functions of curves.

Let $\pi_q(n)$ denote the number of monic, irreducible polynomials of degree $n$ over $\mathbb{F}_q[x]$. The following Prime Polynomial Theorem holds
\begin{equation*}
\pi_q(n) = \frac{1}{n} \sum_{d|n} \mu(d) q^{n/d}.
\label{pnt}
\end{equation*}
We can rewrite the Prime Polynomial Theorem in the form 
\begin{equation*}
\sum_{f \in \mathcal{M}_n} \Lambda(f) = q^n.
\end{equation*}

\subsection{Quadratic Dirichlet $L$-functions over function fields}

For $\textrm{Re}(s)>1$, the zeta function of $\mathbb{F}_q[x]$ is defined by
\[
\zeta_q(s):=\sum_{f\in\mathcal{M}}\frac{1}{|f|^s}=\prod_{P\in \mathcal{P}}\bigg(1-\frac{1}{|P|^s}\bigg)^{-1}.
\]
Since there are $q^n$ monic polynomials of degree $n$, we see that
\[
\zeta_q(s)=\frac{1}{1-q^{1-s}}.
\]
It is sometimes convenient to make the change of variable $u=q^{-s}$, and then write $\mathcal{Z}(u)=\zeta_q(s)$, so that $$\mathcal{Z}(u)=\frac{1}{1-qu}.$$

For $P$ a monic irreducible polynomial, the quadratic residue symbol $\big(\frac{f}{P}\big)\in\{0,\pm1\}$ is defined by
\[
\Big(\frac{f}{P}\Big)\equiv f^{(|P|-1)/2}(\textrm{mod}\ P).
\]
If $Q=P_{1}^{\alpha_1}P_{2}^{\alpha_2}\ldots P_{r}^{\alpha_r}$, then the Jacobi symbol is defined by
\[
\Big(\frac{f}{Q}\Big)=\prod_{j=1}^{r}\Big(\frac{f}{P_j}\Big)^{\alpha_j}.
\]
The Jacobi symbol satisfies the quadratic reciprocity law. That is to say if $A,B\in \mathbb{F}_q[x]$ are relatively prime, monic polynomials, then
\[
\Big(\frac{A}{B}\Big)=(-1)^{(q-1)d(A)d(B)/2}\Big(\frac{B}{A}\Big).
\]
As we are assuming $q\equiv 1(\textrm{mod}\ 4)$, the quadratic reciprocity law gives $\big(\frac{A}{B}\big)=\big(\frac{B}{A}\big)$, a fact we will use throughout the paper.

For $D$ monic, we define the character 
\[
\chi_D(g)=\Big(\frac{D}{g}\Big),
\]
and consider the $L$-function attached to $\chi_D$,
\[
L(s,\chi_D):=\sum_{f\in\mathcal{M}}\frac{\chi_D(f)}{|f|^s}.
\]
With the change of variable $u=q^{-s}$ we have
\begin{equation*}
\mathcal{L}(u,\chi_D):=L(s,\chi_D)=\sum_{f\in\mathcal{M}}\chi_D(f)u^{d(f)}=\prod_{P\in \mathcal{P}}\big(1-\chi_D(P)u^{d(P)}\big)^{-1}.
\end{equation*}
For $D\in\mathcal{H}_{2g+1}$, $\mathcal{L}(u,\chi_D)$ is a polynomial in $u$ of degree $2g$ and it satisfies a functional equation
\begin{equation*}
\mathcal{L}(u,\chi_D)=(qu^2)^g\mathcal{L}\Big(\frac{1}{qu},\chi_D\Big).
\end{equation*}

There is a connection between $\mathcal{L}(u,\chi_D)$ and zeta function of curves. For $D\in\mathcal{H}_{2g+1}$, the affine equation $y^2=D(x)$ defines a projective and connected hyperelliptic curve $C_D$ of genus $g$ over $\mathbb{F}_q$. The zeta function of the curve $C_D$ is defined by
\[
Z_{C_D}(u)=\exp\bigg(\sum_{j=1}^{\infty}N_j(C_D)\frac{u^j}{j}\bigg),
\]
where $N_j(C_D)$ is the number of points on $C_D$ over $\mathbb{F}_q$, including the point at infinity. Weil [\textbf{\ref{W}}] showed that
\[
Z_{C_D}(u)=\frac{P_{C_D}(u)}{(1-u)(1-qu)},
\]
where $P_{C_D}(u)$ is a polynomial of degree $2g$. It is known that $P_{C_D}(u)=\mathcal{L}(u,\chi_D)$ (this was proved in Artin's thesis). The Riemann Hypothesis for curves over function fields was proven by Weil [\textbf{\ref{W}}], so all the zeros of $\mathcal{L}(u,\chi_D)$ are on the circle $|u|=q^{-1/2}$. 

\subsection{Preliminary lemmas}

The first three lemmas are in [\textbf{\ref{F1}}; Lemma 2.2, Proposition 3.1 and Lemma 3.2].

\begin{lemma}\label{L1}
For $f\in\mathcal{M}$ we have
\[
\sum_{D\in\mathcal{H}_{2g+1}}\chi_D(f)=\sum_{C|f^\infty}\sum_{h\in\mathcal{M}_{2g+1-2d(C)}}\chi_f(h)-q\sum_{C|f^\infty}\sum_{h\in\mathcal{M}_{2g-1-2d(C)}}\chi_f(h),
\]
where the summations over $C$ are over monic polynomials $C$ whose prime factors are among the prime factors of $f$.
\end{lemma}

We define the generalized Gauss sum as
\[
G(V,\chi):= \sum_{u (\textrm{mod}\ f)}\chi(u)e\Big(\frac{uV}{f}\Big),
\]
where the exponential was defined in [\textbf{\ref{Hayes}}] as follows. For $a \in  \mathbb{F}_q\big((\frac 1x)\big) $, 
$$ e(a) = e^{ 2 \pi i \text{Tr}_{\mathbb{F}_q / \mathbb{F}_p} (a_1)/p},$$ where $a_1$ is the coefficient of $1/x$ in the Laurent expansion of $a$. 

\begin{lemma}\label{L3}
Let $f\in\mathcal{M}_n$. If $n$ is even then
\[
\sum_{h\in\mathcal{M}_m}\chi_f(h)=\frac{q^m}{|f|}\bigg(G(0,\chi_f)+q\sum_{V\in\mathcal{M}_{\leq n-m-2}}G(V,\chi_f)-\sum_{V\in\mathcal{M}_{\leq n-m-1}}G(V,\chi_f)\bigg),
\]
otherwise
\[
\sum_{h\in\mathcal{M}_m}\chi_f(h)= \frac{q^{m+1/2}} {|f|}\sum_{V\in\mathcal{M}_{n-m-1}}G(V,\chi_f).
\]
\end{lemma}

\begin{lemma}\label{L2}

\begin{enumerate}
\item If $(f,h)=1$, then $G(V, \chi_{fh})= G(V, \chi_f) G(V,\chi_h)$.
\item Write $V= V_1 P^{\alpha}$ where $P \nmid V_1$.
Then 
 $$G(V , \chi_{P^j})= 
\begin{cases}
0 & \mbox{if } j \leq \alpha \text{ and } j \text{ odd,} \\
\varphi(P^j) & \mbox{if }  j \leq \alpha \text{ and } j \text{ even,} \\
-|P|^{j-1} & \mbox{if }  j= \alpha+1 \text{ and } j \text{ even,} \\
\chi_P(V_1) |P|^{j-1/2} & \mbox{if } j = \alpha+1 \text{ and } j \text{ odd, } \\
0 & \mbox{if } j \geq 2+ \alpha .
\end{cases}$$ 
\end{enumerate}
\end{lemma}

\begin{lemma}\label{L5}
For $\ell\in\mathcal{M}$ a square polynomial we have
\[
\frac{1}{| \mathcal{H}_{2g+1}|}\sum_{D \in \mathcal{H}_{2g+1}} \chi_D(\ell)=\prod_{\substack{P\in\mathcal{P}\\P|\ell}}\bigg(1+\frac{1}{|P|}\bigg)^{-1}+O(q^{-2g}).
\]
\end{lemma}
\begin{proof}
See [\textbf{\ref{BF}}; Lemma 3.7].
\end{proof}

We also have the following estimate.
\begin{lemma}[P\'olya--Vinogradov inequality]
\label{nonsq}
For $\ell\in\mathcal{M}$ not a square polynomial, let $\ell= \ell_1 \ell_2^2$ with $\ell_1$ square-free. Then we have
$$\bigg|  \sum_{D \in \mathcal{H}_{2g+1}} \chi_D(\ell) \bigg| \ll_\varepsilon q^{g} | \ell_1|^{\varepsilon}.$$
\end{lemma}
\begin{proof}
As in the proof of Lemma $2.2$ in [\textbf{\ref{F1}}], using Perron's formula we have
$$ \sum_{D \in \mathcal{H}_{2g+1}} \chi_D( \ell) = \frac{1}{2 \pi i} \oint_{|u|=r}\mathcal{L} (u,\chi_{\ell})\prod_{P | \ell} \Big(1-u^{2d(P)}\Big)^{-1} \frac{(1-qu^2) du}{  u^{2g+2}},$$
where we pick $r = q^{-1/2}$. If we write $\ell= \ell_1 \ell_2^2$ with $\ell_1$ square-free, then
$$ \mathcal{L}(u,\chi_{\ell})= \mathcal{L}(u,\chi_{\ell_1}) \prod_{\substack{P \nmid \ell_1 \\ P | \ell_2}} \Big(1-u^{d(P)}\chi_{\ell_1}(P) \Big).$$
Now we use the Lindel\"{o}f bound  for $\mathcal{L}(u,\chi_{\ell_1})$ (see Theorem $3.4$ in [\textbf{\ref{altug}}]),
$$ \mathcal{L}(u,\chi_{\ell_1}) \ll |\ell_1|^{\varepsilon},$$ in the integral above and the conclusion follows.
\end{proof}

\begin{lemma}[Mertens' theorem]
\label{mertens}
We have
$$ \prod_{d(P) \leq X} \bigg( 1-\frac{1}{|P|} \bigg)^{-1} = e^{\gamma} X + O(1),$$ where $\gamma$ is the Euler constant.
\end{lemma}
\begin{proof}
A more general version of Mertens' estimate was proved in [\textbf{\ref{R}}; Theorem 3]. Here we give a simpler proof in the above form for completeness.

Using the Prime Polynomial Theorem,
$$ \sum_{d(P) \leq X} \frac{d(P)}{|P|} = X+O(1),$$ and hence by partial summation, we get that
$$ \sum_{d(P) \leq X} \frac{1}{|P|} = \log X + c + O \Big( \frac{1}{X} \Big)$$ for some constant $c$.  Then
\begin{align*}
\sum_{d(P) \leq X} \log &\bigg( 1- \frac{1}{|P|} \bigg)^{-1} = \sum_{d(P) \leq X} \frac{1}{|P|} + \sum_{d(P) \leq X} \sum_{j=2}^{\infty} \frac{1}{j|P|^j}\\
&= \log X +c +  \sum_{P\in\mathcal{P}} \sum_{j=2}^{\infty} \frac{1}{j|P|^j} - \sum_{d(P) >X}  \sum_{j=2}^{\infty} \frac{1}{j|P|^j} + O \Big( \frac{1}{X} \Big) \\
&= \log X + C +O \Big( \frac{1}{X} \Big),
\end{align*} where $C = c +  \sum_{P\in\mathcal{P}} \sum_{j=2}^{\infty} \frac{1}{j|P|^j}$. Exponentiating and using the fact that for $x<1$, $e^x=1+O(x)$, we get that
$$ \prod_{d(P) \leq X} \bigg( 1-\frac{1}{|P|} \bigg)^{-1} = e^{C} X + O(1),$$ and it remains to show that $C=\gamma$. 

Now by the Prime Polynomial Theorem,
$$ \sum_{\substack{f\in\mathcal{M}\\ d(f) \leq X}} \frac{\Lambda(f)}{|f| d(f)} = \sum_{n \leq X} \frac{1}{n} = \log X + \gamma + O \Big( \frac{1}{X} \Big).$$
Combining the formulas above, we also have that
\begin{align*}
\sum_{\substack{f\in\mathcal{M}\\ d(f) \leq X}} \frac{\Lambda(f)}{|f| d(f)} &= \sum_{d(P) \leq X} \frac{1}{|P|} + \sum_{P\in\mathcal{P}}\sum_{2\leq j\leq X/d(P)} \frac{1}{j |P|^j} \\
&= \log X + c+ \sum_{P\in\mathcal{P}} \sum_{j=2}^{\infty}  \frac{1}{j|P|^j} - \sum_{j=2}^{\infty} \sum_{d(P) > X/j} \frac{1}{j|P|^j}  + O \Big( \frac{1}{X} \Big) \\
&= \log X + C + O \Big( \frac{1}{X}\Big).
\end{align*}
Using the previous two identities, it follows that $C = \gamma$, which finishes the proof.
\end{proof}
%{\color{blue} There's a paper of Rosen (A generalization of Mertens' theorem) in which he probably proves something similar, but I couldn't find the paper online. If you can find it and check his result, we could quote it here instead of writing down the proof.}
%\begin{lemma}[The Weil bound]\label{Weil}
%For $V\in\mathcal{M}$ not a perfect square we have
%$$ \left|  \sum_{P \in \mathcal{P}_n} \chi_V(P) \right| \ll \frac{d(V)}{n} q^{n/2}.$$
%\label{sumprimes}
%\end{lemma}
%\begin{proof}
%See [\textbf{\ref{R}}; Equation $2.5$].
%\end{proof}

%{\color{blue} In [\textbf{\ref{AK}}], Lemma $6.4$, Andrade and Keating prove the weaker bound of $q^{g+1/2} 2^{d(\ell)-1}$. I think our Lemma above is correct.}

\section{Hybrid Euler-Hadamard product}

We start with an explicit formula.

\begin{lemma}\label{explicit}
Let $u(x)$ be a real, non-negative, $C^{\infty}$ function with mass $1$ and compactly supported on $[q,q^{1+1/X}]$. Let $v(t)=\int_{t}^{\infty}u(x)dx$ and let $\widetilde{u}$ be the Mellin transform of $u$. Then for $s$ not a zero of $L(s,\chi_D)$ we have
\begin{eqnarray*}
&&-\frac{L'}{L}(s,\chi_D)=\sum_{f\in\mathcal{M}}\frac{(\log q)\Lambda(f)\chi_D(f)}{|f|^s}v\big(q^{d(f)/ X}\big)-\sum_{\rho}\frac{\widetilde{u}\big(1-(s-\rho) X\big)}{s-\rho},
\end{eqnarray*}
where the sum over $\rho$ runs over all the zeros of $L(s,\chi_D)$.
\end{lemma}

This lemma can be proved in a familiar way [\textbf{\ref{BH}}], beginning with the integral
\begin{equation*}
-\frac{1}{2\pi i}\int_{(c)}\frac{L'}{L}(s+z,\chi_D)\widetilde{u}(1+zX)\frac{dz}{z},
\end{equation*}
where $c=\max\{2,2-\textrm{Re}(s)\}$.

Following the arguments in [\textbf{\ref{GHK}}], we can integrate the formula in Lemma \ref{explicit} to give a formula
for $L(s,\chi_D)$: for $s$ not equal to one of the zeros and $\textrm{Re}(s)\geq0$ we have
\begin{eqnarray}\label{explicitintegrate}
L(s,\chi_D)=\exp\bigg(\sum_{f\in\mathcal{M}}\frac{\Lambda(f)\chi_D(f)}{|f|^sd(f)}v\big(q^{d(f)/ X}\big)\bigg)Z_{X}(s,\chi_D).
\end{eqnarray}
To remove the former restriction on $s$, we note that we may interpret $\exp\big(-U(z)\big)$ to be asymptotic to $Cz$ for some constant $C$ as $z\rightarrow0$, so both sides of \eqref{explicitintegrate} vanish at the zeros. Thus \eqref{explicitintegrate} holds for all $\textrm{Re}(s)\geq0$. Furthermore, since $v(q^{d(f)/X})=1$ for $d(f)\leq X$ and $v(q^{d(f)/X})=0$ for $d(f)\geq X+1$, the first factor in \eqref{explicitintegrate} is precisely $P_{X}(s,\chi_D)$, and that completes the proof of Theorem \ref{HEH}.

\section{Moments of the partial Euler product}\label{PXchi}

Recall that
\begin{equation*}
P_{X}(s,\chi_D)=\exp\bigg( \sum_{\substack{f\in\mathcal{M}\\d(f)\leq X}}\frac{\Lambda(f)\chi_D(f)}{|f|^{s}d(f)}\bigg).
\end{equation*}
We first show that we can approximate $P_X(s,\chi_D)^k$ by
\begin{equation*}
P_{k,X}^{*}(s,\chi_D)=\prod_{d(P)\leq X/2}\bigg( 1-\frac{\chi_D(P)}{|P|^{s}}\bigg) ^{-k}\prod_{X/2<d(P)\leq X}\bigg( 1+\frac{k\chi_{D}(P)}{|P|^s}+\frac{k^2\chi_{D}(P)^2}{2|P|^{2s}}\bigg)
\end{equation*}
for any $k\in\mathbb{R}$.

\begin{lemma}\label{PP*}
For any $k\in\mathbb{R}$ we have
\begin{equation*}
P_{X}(s,\chi_D)^{k}=\Big(1+O_{k}\big(q^{-X/6}/X\big)\Big)P_{k,X}^{*}(s,\chi_D)
\end{equation*}
uniformly for $\emph{Re}(s)=\sigma\geq1/2$.
\end{lemma}
\begin{proof}
For any $P\in\mathcal{P}$ we let $N_{P}=\lfloor X/d(P)\rfloor$, the integer part of $X/d(P)$. Then we have
\begin{equation*}
P_{X}(s,\chi_D)^{k}=\exp\bigg(k\sum_{d(P)\leq X}\sum_{1\leq j\leq N_{P}}\frac{\chi_D(P^j)}{j|P|^{js}}\bigg)
\end{equation*}
and
\begin{align*}
P_{k,X}^{*}(s,\chi_D)=&\exp\bigg(k\sum_{d(P)\leq X/2}\sum_{j=1}^{\infty}\frac{\chi_D(P)^j}{j|P|^{js}}+\sum_{X/2<d(P)\leq X}\frac{k\chi_D(P)}{|P|^{s}}\\
&\qquad\qquad\qquad\qquad\qquad\qquad\qquad\qquad\qquad+O_k\bigg(\sum_{X/2<d(P)\leq X}\frac{1}{|P|^{3\sigma}}\bigg)\bigg).
\end{align*}
We note that $N_{P}=1$ for $X/2<d(P)\leq X$, so
\begin{equation*}
P_{X}(s,\chi_D)^{k}P_{k,X}^{*}(s,\chi_D)^{-1}=\exp\bigg(-k\sum_{d(P)\leq X/2}\sum_{j>N_{P}}\frac{\chi_D(P)^j}{j|P|^{js}}+O_k\bigg(\sum_{X/2<d(P)\leq X}\frac{1}{|P|^{3\sigma}}\bigg)\bigg).
\end{equation*}
The expression in the exponent is
\begin{eqnarray*}
&\ll_k&\sum_{d(P)\leq X/2}\frac{1}{|P|^{\sigma(N_{P}+1)}}+\sum_{X/2<d(P)\leq X}\frac{1}{|P|^{3\sigma}}\nonumber\\
&\ll_k&\sum_{j=2}^{X}\sum_{X/(j+1)<d(P)\leq X/j}\frac{1}{|P|^{(j+1)/2}}+\sum_{X/2<d(P)\leq X}\frac{1}{|P|^{3/2}}\\
&\ll_k&\sum_{j=2}^{X}\frac{jq^{-(j-1)X/2(j+1)}}{X}+\frac{q^{-X/4}}{X}\ll_k\frac{q^{-X/6}}{X}.
\end{eqnarray*}
Hence $P_{X}(s,\chi_D)^{k}P_{k,X}^{*}(s,\chi_D)^{-1}=1+O_{k}\big(q^{-X/6}/X\big)$ as claimed.
\end{proof}

Next we write $P_{k,X}^{*}(s,\chi_{D})$ as a Dirichlet series
\begin{equation*}
\sum_{\ell\in\mathcal{M}}\frac{\alpha_{k}(\ell)\chi_{D}(\ell)}{|\ell|^s}=\prod_{d(P)\leq X/2}\bigg( 1-\frac{\chi_D(P)}{|P|^s}\bigg) ^{-k}\prod_{X/2<d(P)\leq X}\bigg(1+\frac{k\chi_{D}(P)}{|P|^s}+\frac{k^2\chi_{D}(P)^2}{2|P|^{2s}}\bigg) .
\end{equation*}
We note that $\alpha_{k}(\ell)\in\mathbb{R}$, and if we denote by $S(X)$ the set of $X$-smooth polynomials, i.e.
\begin{displaymath}
S(X)=\{\ell\in\mathcal{M}:P|\ell\rightarrow d(P)\leq X\},
\end{displaymath}
then $\alpha_{k}(\ell)$ is multiplicative, and $\alpha_{k}(\ell)=0$ if $\ell\notin S(X)$. We also have $0\leq\alpha_{k}(\ell)\leq\tau_{|k|}(\ell)$ for all $\ell\in\mathcal{M}$. Moreover, $\alpha_{k}(\ell)=\tau_{k}(\ell)$  if $\ell\in S(X/2)$, and $\alpha_{k}(P)=k$ and $\alpha_{k}(P^2)=k^2/2$ for all $P\in\mathcal{P}$ with $X/2<d(P)\leq X$. 

We now truncate the series, for $s=1/2$, at $d(\ell)\leq \vartheta g$. From the Prime Polynomial Theorem we have
\begin{align}\label{truncation}
\sum_{\substack{\ell\in S(X)\\ d(\ell)> \vartheta g}}\frac{\alpha_{k}(\ell)\chi_{D}(\ell)}{|\ell|^{1/2}}&\leq \sum_{\ell\in S(X)}\frac{\tau_{|k|}(\ell)}{|\ell|^{1/2}}\Big(\frac{|\ell|}{q^{\vartheta g}}\Big)^{c/4}=q^{-c\vartheta g/4}\prod_{d(P)\leq X}\bigg(1-\frac{1}{|P|^{(2-c)/4}}\bigg)^{-|k|}\nonumber\\
&\ll q^{-c\vartheta g/4}\exp\bigg(O_k\Big(\sum_{d(P)\leq X}\frac{1}{|P|^{(2-c)/4}}\Big)\bigg)\nonumber\\
&\ll q^{-c\vartheta g/4}\exp\bigg(O_k\Big(\frac{q^{(2+c)X/4}}{X}\Big)\bigg)\ll_\varepsilon q^{-c\vartheta g/4+\varepsilon g},
\end{align}
as $X\leq (2-c)\log g/\log q$. Hence
\begin{equation}\label{P*}
P_{k,X}^{*}(\chi_{D}):=P_{k,X}^{*}(\tfrac12,\chi_{D})=\sum_{\substack{\ell\in S(X)\\ d(\ell)\leq\vartheta g}}\frac{\alpha_{k}(\ell)\chi_{D}(\ell)}{|\ell|^{1/2}}+O_\varepsilon(q^{-c\vartheta g/4+\varepsilon g})
\end{equation}
for all $k\in\mathbb{R}$ and $\vartheta>0$, and it follows that
\begin{equation*}
\Big\langle P_{k,X}^{*}(\chi_{D}) \Big\rangle_{\mathcal{H}_{2g+1}}=\frac{1}{|\mathcal{H}_{2g+1}|}\sum_{\substack{\ell\in S(X)\\ d(\ell)\leq\vartheta g}}\frac{\alpha_{k}(\ell)}{|\ell|^{1/2}}\sum_{D\in\mathcal{H}_{2g+1}} \chi_{D}(\ell)+O_\varepsilon\big(q^{-c\vartheta g/4+\varepsilon g}\big).
\end{equation*}

We first consider the contribution of the terms with $\ell=\square$. Denote this by $I(\ell=\square)$. By Lemma \ref{L5}, 
\begin{displaymath}
I\big(\ell=\square\big)=\sum_{\substack{\ell\in S(X)\\ d(\ell)\leq\vartheta g/2}}\frac{\alpha_{k}(\ell^2)}{|\ell|}\prod_{\substack{P\in\mathcal{P}\\P|\ell}}\bigg(1+\frac{1}{|P|}\bigg)^{-1}+O_\varepsilon\big(q^{-2g+\varepsilon g}\big).
\end{displaymath}
The sum can be extended to all $\ell\in S(X)$ as, like in \eqref{truncation},
\begin{displaymath}
\sum_{\substack{\ell\in S(X)\\ d(\ell)>\vartheta g/2}}\frac{\alpha_{k}(\ell^2)}{|\ell|}\prod_{\substack{P\in\mathcal{P}\\P|\ell}}\bigg(1+\frac{1}{|P|}\bigg)^{-1}\ll_\varepsilon q^{-\vartheta g/4+\varepsilon g}.
\end{displaymath}
So, using the multiplicativity of $\alpha_k(\ell)$ and Lemma \ref{mertens},
\begin{align*}
I\big(&\ell=\square\big)=\prod_{d(P)\leq X}\bigg(1+\sum_{j=1}^{\infty}\frac{\alpha_{k}(P^{2j})}{|P|^{j-1}(1+|P|)}\bigg)+O_\varepsilon\big(q^{-\vartheta g/4+\varepsilon g}\big)\\
&=\prod_{d(P)\leq X/2}\bigg(1+\sum_{j=1}^{\infty}\frac{\tau_{k}(P^{2j})}{|P|^{j-1}(1+|P|)}\bigg)\prod_{X/2<d(P)\leq X}\bigg(1+\frac{k^2}{2(1+|P|)}+O_\varepsilon\big(|P|^{-2+\varepsilon}\big)\bigg)\\
&\qquad\qquad+O_\varepsilon\big(q^{-\vartheta g/4+\varepsilon g}\big)\\
&=\Big(1+O\big(q^{-X/2}/X\big)\Big)\prod_{d(P)\leq X/2}\bigg[\bigg(1-\frac{1}{|P|}\bigg)^{k(k+1)/2}\bigg(1+\sum_{j=1}^{\infty}\frac{\tau_{k}(P^{2j})}{|P|^{j-1}(1+|P|)}\bigg)\bigg]\\
&\qquad\qquad\prod_{d(P)\leq X/2}\bigg(1-\frac{1}{|P|}\bigg)^{-k(k+1)/2}\prod_{X/2<d(P)\leq X}\bigg(1-\frac{1}{|P|}\bigg)^{-k^2/2}+O_\varepsilon\big(q^{-\vartheta g/4+\varepsilon g}\big)\\
&=\bigg(1+O\Big(\frac 1X\Big)\bigg)2^{-k/2}\mathcal{A}_k\big(e^\gamma X\big)^{k(k+1)/2}+O_\varepsilon\big(q^{-\vartheta g/4+\varepsilon g}\big).
\end{align*}

Now we consider the contribution from $\ell \neq \square$, which we denote by $I ( \ell \neq \square)$. Using Lemma \ref{nonsq} we have that
$$I(\ell \neq \square) \ll q^{-g} \sum_{\ell \in S(X)} \frac{ \tau_{|k|}(\ell)}{|\ell|^{1/2-\varepsilon}}.$$ As in \eqref{truncation},
\begin{align*}
 \sum_{\ell \in S(X)} \frac{ \tau_{|k|}(\ell)}{|\ell|^{1/2-\varepsilon}}& \ll \prod_{d(P) \leq X} \bigg( 1- \frac{1}{|P|^{1/2-\varepsilon}} \bigg)^{-|k|}\ll \exp \bigg( O_k\Big( \sum_{d(P) \leq X} \frac{1}{|P|^{1/2 -\varepsilon}}\Big) \bigg)\\
&\ll \exp \bigg( O_k\Big(\frac{q^{(1/2+\varepsilon)X}}{X}\Big) \bigg)\ll_\varepsilon q^{\varepsilon g}.
\end{align*}
Hence
$$I( \ell \neq \square) \ll_\varepsilon q^{-g+\varepsilon g},$$
and we obtain the theorem.

\section{Twisted moments of $L(\frac12,\chi_D)$}
\label{twist}
In this section, we are interested in the $k$-th twisted moment
\[
I_k(\ell;g)=\frac{1}{|\mathcal{H}_{2g+1}|}\sum_{D\in\mathcal{H}_{2g+1}}L(\tfrac12,\chi_D)^k\chi_D(\ell).
\]

We first recall the approximate functional equation,
\[
L(\tfrac12,\chi_D)^k=\sum_{f\in\mathcal{M}_{kg}}\frac{\tau_k(f)\chi_D(f)}{|f|^{1/2}}+\sum_{f\in\mathcal{M}_{kg-1}}\frac{\tau_k(f)\chi_D(f)}{|f|^{1/2}}
\]
for any $k\in\mathbb{N}$. So
\[
I_k(\ell;g)=S_{k}(\ell;kg)+S_{k}(\ell;kg-1),
\]
where
\[
S_k(\ell;N)=\frac{1}{|\mathcal{H}_{2g+1}|}\sum_{f\in\mathcal{M}_{\leq N}}\frac{\tau_k(f)}{|f|^{1/2}}\sum_{D\in\mathcal{H}_{2g+1}}\chi_D(f\ell)
\]
for $N\in\{kg,kg-1\}$. Define
\begin{align*}
S_{k,1}(\ell;N)=\frac{1}{|\mathcal{H}_{2g+1}|}\sum_{f\in\mathcal{M}_{\leq N}}\frac{\tau_k(f)}{|f|^{1/2}}\sum_{C|(f\ell)^\infty}\sum_{h\in\mathcal{M}_{2g+1-2d(C)}}\chi_{f\ell}(h)
\end{align*}
and
\[
S_{k,2}(\ell;N)=\frac{1}{|\mathcal{H}_{2g+1}|}\sum_{f\in\mathcal{M}_{\leq N}}\frac{\tau_k(f)}{|f|^{1/2}}\sum_{C|(f\ell)^\infty}\sum_{h\in\mathcal{M}_{2g-1-2d(C)}}\chi_{f\ell}(h)
\]
so that, in view of Lemma \ref{L1},
\[
S_k(\ell;N)=S_{k,1}(\ell;N)-qS_{k,2}(\ell;N).
\]

We further write 
\[
S_{k,1}(\ell;N)=S_{k,1}^{\textrm{o}}(\ell;N)+S_{k,1}^{\textrm{e}}(\ell;N)
\]
according to whether the degree of the product $f\ell$ is even or odd, respectively. Lemma \ref{L3} and Lemma \ref{L2} then lead to
\begin{align}
S_{k,1}^{\textrm{o}}(\ell;N)=\frac{q^{3/2}}{(q-1)|\ell|}\sum_{\substack{f\in\mathcal{M}_{\leq N}\\d(f\ell)\ \textrm{odd}}}\frac{\tau_k(f)}{|f|^{3/2}}\sum_{\substack{C|(f\ell)^\infty\\d(C)\leq g}}\frac{1}{|C|^2}\sum_{V\in\mathcal{M}_{d(f\ell)-2g-2+2d(C)}}G(V,\chi_{f\ell})
\label{s1odd}
\end{align}
and
\[
S_{k,1}^{\textrm{e}}(\ell;N)=M_{k,1}(\ell;N)+S_{k,1}^{\textrm{e}}(\ell;N;V\ne0),
\]
where
\begin{align}\label{M}
M_{k,1}(\ell;N)=\frac{q}{(q-1)|\ell|}\sum_{\substack{f\in\mathcal{M}_{\leq N}\\f\ell=\square}}\frac{\tau_k(f)\varphi(f\ell)}{|f|^{3/2}}\sum_{\substack{C|(f\ell)^\infty\\d(C)\leq g}}\frac{1}{|C|^2}
\end{align}
and
\begin{align}\label{Ske}
S_{k,1}^{\textrm{e}}(\ell;N;V\ne0)&=\frac{q}{(q-1)|\ell|}\sum_{\substack{f\in\mathcal{M}_{\leq N}\\d(f\ell)\ \textrm{even}}}\frac{\tau_k(f)}{|f|^{3/2}}\sum_{\substack{C|(f\ell)^\infty\\d(C)\leq g}}\frac{1}{|C|^2}\\
&\bigg(q\sum_{V\in\mathcal{M}_{\leq d(f\ell)-2g-3+2d(C)}}G(V,\chi_{f\ell})-\sum_{V\in\mathcal{M}_{\leq d(f\ell)-2g-2+2d(C)}}G(V,\chi_{f\ell})\bigg).\nonumber
\end{align}
We also decompose
\[
S_{k,1}^{\textrm{e}}(\ell;N;V\ne0)=S_{k,1}^{\textrm{e}}(\ell;N;V=\square)+S_{k,1}^{\textrm{e}}(\ell;N;V\ne\square)
\]
correspondingly to whether $V$ is a square or not.

We treat $S_{k,2}(\ell;N)$ similarly and define the functions $S_{k,2}^{\textrm{o}}(\ell;N)$, $M_{k,2}(\ell;N)$, $S_{k,2}^{\textrm{e}}(\ell;N;V=\square)$ and $S_{k,2}^{\textrm{e}}(\ell;N;V\ne\square)$ in the same way. Further denote 
\[
M_{k}(\ell;N)=M_{k,1}(\ell;N)-qM_{k,2}(\ell;N),\qquad M_{k}(\ell)=M_{k}(\ell;kg)+M_{k}(\ell;kg-1)
\] 
and
$$S_{k}^{\textrm{e}}(\ell;V=\square)=S_{k}^{\textrm{e}}(\ell;kg;V=\square)+S_{k}^{\textrm{e}}(\ell;kg-1;V=\square),$$
where $$S_{k}^{\textrm{e}}(\ell;N;V=\square)=S_{k,1}^{\textrm{e}}(\ell;N;V=\square)-qS_{k,2}^{\textrm{e}}(\ell;N;V=\square).$$ 

We shall next consider $M_{k}(\ell)$. The term $S_{k}^{\textrm{e}}(\ell;V=\square)$ also contributes to the main term and will be evaluated in Section \ref{Vsquare}. We will see that that combines nicely with the contribution from $M_{k}(\ell)$ for $k=1$. For the terms $S_{k,1}^{\textrm{o}}(\ell;N)$ and $S_{k,2}^{\textrm{o}}(\ell;N)$, we note that the summations over $V$ are over odd degree polynomials, so $V\ne\square$ in these cases. Let $S_{k}^{\textrm{o}}(\ell;N) = S_{k,1}^{\textrm{o}}(\ell;N)-qS_{k,2}^{\textrm{o}}(\ell;N)$, $S_{k}^{\textrm{e}}(\ell;N;V\ne\square)=S_{k,1}^{\textrm{e}}(\ell;N;V\ne\square)-qS_{k,2}^{\textrm{e}}(\ell;N;V\ne\square)$ and 
\begin{equation}\label{Sknonsquare}
S_k(\ell;N;V \neq \square) = S_{k}^{\textrm{e}}(\ell;N;V\ne\square) +S_{k}^{\textrm{o}}(\ell;N)
\end{equation} be the total contribution from $V \neq \square$.
We will bound $S_k(\ell;N;V \neq \square)$ in Section \ref{Vnonsquare}.

\subsection{Evaluate $M_{k}(\ell)$}
\label{main}
We first note that the sum over $C$ in \eqref{M} can be extended to all $C|(f\ell)^\infty$ with the cost of an error of size $O_\varepsilon(q^{N/2-2g+\varepsilon g})=O_\varepsilon\big(q^{(k-4)g/2+\varepsilon g}\big)$, as
\begin{equation}\label{Cextend}
\sum_{\substack{C|(f\ell)^\infty\\d(C)> g}}\frac{1}{|C|^2}\ll \sum_{C|(f\ell)^\infty}\frac{1}{|C|^2}\Big(\frac{|C|}{q^{g}}\Big)^{2-\varepsilon}=q^{-2g+\varepsilon g}\prod_{P|f\ell}\bigg(1-\frac{1}{|P|^\varepsilon}\bigg)^{-1}.
\end{equation}
So
\begin{align*}
M_{k,1}(\ell;N)&=\frac{q}{(q-1)|\ell|}\sum_{\substack{f\in\mathcal{M}_{\leq N}\\f\ell=\square}}\frac{\tau_k(f)\varphi(f\ell)}{|f|^{3/2}}\prod_{P|f\ell}\bigg(1-\frac{1}{|P|^2}\bigg)^{-1}+O_\varepsilon\big(q^{(k-4)g/2+\varepsilon g}\big)\\
&=\frac{q}{(q-1)}\sum_{\substack{f\in\mathcal{M}_{\leq N}\\f\ell=\square}}\frac{\tau_k(f)}{|f|^{1/2}}\prod_{P|f\ell}\bigg(1+\frac{1}{|P|}\bigg)^{-1}+O_\varepsilon\big(q^{(k-4)g/2+\varepsilon g}\big).
\end{align*}
The condition $f\ell=\square$ implies that $f=f_1^2\ell_1$ for some $f_1\in\mathcal{M}$. Hence
\begin{align*}
M_{k,1}(\ell;N)=&\frac{q}{(q-1)|\ell_1|^{1/2}}\sum_{2d(f)\leq N-d(\ell_1)}\frac{\tau_k(f^2\ell_1)}{|f|}\prod_{P|f\ell_1\ell_2}\bigg(1+\frac{1}{|P|}\bigg)^{-1}+O_\varepsilon\big(q^{(k-4)g/2+\varepsilon g}\big).
\end{align*}

We are going to use an analogue of the Perron formula in the following form
\[
\sum_{2n\leq N}a(n)=\frac{1}{2\pi i}\int_{|u|=r}\Big(\sum_{n=0}^{\infty}a(n)u^{2n}\Big)\frac{du}{u^{N+1}(1-u)},
\]
provided that the power series $\sum_{n=0}^{\infty}a(n)u^n$ is absolutely convergent in $|u|\leq r<1$. Hence
\begin{align*}
M_{k,1}(\ell;N)=\frac{q}{(q-1)|\ell_1|^{1/2}}\frac{1}{2\pi i}\oint_{|u|=r}\frac{\mathcal{F}_k(u)du}{u^{N-d(\ell_1)+1}(1-u)}+O_\varepsilon\big(q^{(k-4)g/2+\varepsilon g}\big)
\end{align*}
for any $r<1$, where
\begin{align*}
\mathcal{F}_k(u)&=\sum_{f\in\mathcal{M}}\frac{\tau_k(f^2\ell_1)}{|f|}\prod_{P|f\ell_1\ell_2}\bigg(1+\frac{1}{|P|}\bigg)^{-1}u^{2d(f)}.
\end{align*}
Now by multiplicativity we have
\[
\mathcal{F}_k(u)=\eta_k(\ell;u)\mathcal{Z}\Big(\frac{u^2}{q}\Big)^{k(k+1)/2},
\]
where
\begin{equation}\label{eta}
\eta_k(\ell;u)=\prod_{P\in\mathcal{P}}\mathcal{A}_{k,P}(u)\prod_{P|\ell_1}\mathcal{B}_{k,P}(u)\prod_{\substack{P\nmid\ell_1\\P|\ell_2}}\mathcal{C}_{k,P}(u)
\end{equation}
with
\begin{align*}
&\mathcal{A}_{k,P}(u)=\bigg(1-\frac{u^{2d(P)}}{|P|}\bigg)^{k(k+1)/2}\bigg(1+\bigg(1+\frac{1}{|P|}\bigg)^{-1}\sum_{j=1}^{\infty}\frac{\tau_{k}(P^{2j})}{|P|^j}u^{2jd(P)}\bigg),\\
&\mathcal{B}_{k,P}(u)=\bigg(\sum_{j=0}^{\infty}\frac{\tau_k(P^{2j+1})}{|P|^j}u^{2jd(P)}\bigg)\bigg(1+\frac{1}{|P|}+\sum_{j=1}^{\infty}\frac{\tau_{k}(P^{2j})}{|P|^j}u^{2jd(P)}\bigg)^{-1}
\end{align*} 
and
\begin{align*}
\mathcal{C}_{k,P}(u)=\bigg(\sum_{j=0}^{\infty}\frac{\tau_k(P^{2j})}{|P|^j}u^{2jd(P)}\bigg)\bigg(1+\frac{1}{|P|}+\sum_{j=1}^{\infty}\frac{\tau_{k}(P^{2j})}{|P|^j}u^{2jd(P)}\bigg)^{-1}.
\end{align*}
Thus,
\begin{align*}
M_{k,1}(\ell;N)=&\frac{q}{(q-1)|\ell_1|^{1/2}}\frac{1}{2\pi i}\oint_{|u|=r}\\
&\qquad\qquad\frac{\eta_k(\ell;u)du}{u^{N-d(\ell_1)+1}(1-u)^{k(k+1)/2+1}(1+u)^{k(k+1)/2}}+O_\varepsilon\big(q^{(k-4)g/2+\varepsilon g}\big).
\end{align*}

Similarly, 
\begin{align*}
M_{k,2}(\ell;N)=&\frac{1}{q(q-1)|\ell_1|^{1/2}}\frac{1}{2\pi i}\oint_{|u|=r}\\
&\qquad\qquad\frac{\eta_k(\ell;u)du}{u^{N-d(\ell_1)+1}(1-u)^{k(k+1)/2+1}(1+u)^{k(k+1)/2}}+O_\varepsilon\big(q^{(k-4)g/2+\varepsilon g}\big),
\end{align*}
and hence we obtain that
\begin{align*}
M_{k}(\ell)=&\frac{1}{|\ell_1|^{1/2}}\frac{1}{2\pi i}\oint_{|u|=r}\frac{\eta_k(\ell;u)du}{u^{kg-d(\ell_1)+1}(1-u)^{k(k+1)/2+1}(1+u)^{k(k+1)/2-1}}+O_\varepsilon\big(q^{(k-4)g/2+\varepsilon g}\big).
\end{align*}

As discussed in [\textbf{\ref{F23}}, \textbf{\ref{F1}}], $\eta_k(\ell;u)$ has an analytic continuation to the region $|u|\leq R_k=q^{\vartheta_k}$ for $1\leq k\leq 3$, where $\vartheta_1=1-\varepsilon$, $\vartheta_2=1/2-\varepsilon$ and $\vartheta_3=1/3-\varepsilon$. We then move the contour of integration to $|u|=R_k$, encountering a pole of order $k(k+1)/2+1$ at $u=1$ and a pole of order $k(k+1)/2-1$ at $u=-1$. In doing so we get
\begin{align*}
M_{k}(\ell)=&\frac{1}{|\ell_1|^{1/2}}\frac{1}{2\pi i}\oint_{|u|=R_k}\frac{\eta_k(\ell;u)du}{u^{kg-d(\ell_1)+1}(1-u)^{k(k+1)/2+1}(1+u)^{k(k+1)/2-1}}\\
&\qquad\qquad-\frac{1}{|\ell_1|^{1/2}}\textrm{Res}(u=1) -\frac{1}{|\ell_1|^{1/2}}\textrm{Res}(u=-1)+O_\varepsilon\big(q^{(k-4)g/2+\varepsilon g}\big).
\end{align*}

We now evaluate the residue of the pole at $u=1$ and $u=-1$. We have
\begin{align*}
&\eta_k(\ell;u)=\eta_k(\ell;1)\sum_{j\geq0}\frac{1}{j!}\frac{\partial_u^j\eta_k}{\eta_k}(\ell;1)(u-1)^j,\\
&u^{-(kg-d(\ell_1)+1)}=1-\big(kg-d(\ell_1)+1\big)(u-1)+\ldots,\\
&\frac{1}{(1+u)^{k(k+1)/2-1}}=\frac{1}{2^{k(k+1)/2-1}}-\frac{k(k+1)/2-1}{2^{k(k+1)/2}}(u-1)\ldots.
\end{align*}
Similar expressions hold for the Taylor expansions around $u=-1$. So, using the fact that $\eta_k(\ell;u)$ is even,
\begin{align*}
\textrm{Res}(u=1)+\textrm{Res}&(u=-1)\\
&=\,- \frac{\eta_k(\ell;1)}{2^{k(k+1)/2-1}\big(k(k+1)/2\big)!}\sum_{j=0}^{k(k+1)/2}\frac{\partial_u^j\eta_k}{\eta_k}(\ell;1)P_{k,j}\big(kg-d(\ell_1)\big),
\end{align*}
where $P_{k,j}$'s are some explicit polynomials of degrees $k(k+1)/2-j$ for all $0\leq j\leq k(k+1)/2$, and the leading coefficient of $P_{k,0}$ is $1$.

We note that
\begin{align*}
&\eta_1(\ell;\pm1)=\mathcal{A}_1\prod_{P|\ell}\bigg(1+\frac{1}{|P|}-\frac{1}{|P|^2}\bigg)^{-1},\\
&\eta_2(\ell;\pm1)=\frac{\mathcal{A}_2\tau(\ell_1)|\ell_1|}{\sigma(\ell_1)}\prod_{P|\ell}\bigg(1+\frac{1}{|P|}\bigg)\bigg(1+\frac{2}{|P|}-\frac{2}{|P|^2}+\frac{1}{|P|^3}\bigg)^{-1},\\
&\eta_3(\ell;\pm1)=\mathcal{A}_3\prod_{P|\ell}\bigg(1+\frac{3}{|P|}\bigg)\bigg(1+\frac{4}{|P|}-\frac{3}{|P|^2}+\frac{3}{|P|^3}-\frac{1}{|P|^4}\bigg)^{-1}\prod_{P|\ell_1}\frac{1+3|P|}{3+|P|},
\end{align*}
where $\mathcal{A}_k$'s are as in Conjecture \ref{ak}, and $\sigma(\ell_1)= \sum_{d | \ell_1} |d| = \prod_{P | \ell_1} (1+|P|)$ is the sum of divisors function.
Moreover, differentiating \eqref{eta} $j$ times we see that
\[
\frac{\partial_u^j\eta_k}{\eta_k}(\ell;1)=\sum_{j_1,j_2=0}^{j}c_{j,j_1,j_2}\bigg(\sum_{P|\ell_1}\frac{D_{1,j,j_1}(P)d(P)^{j_1}}{|P|}\bigg)\bigg(\sum_{\substack{P\nmid\ell_1\\P|\ell_2}}\frac{D_{2,j,j_2}(P)d(P)^{j_2}}{|P|}\bigg)
\]
for some absolute constants $c_{j,j_1,j_2}$ and $D_{1,j,j_1}(P)\ll_{j,j_1}1$, $D_{2,j,j_2}(P)\ll_{j,j_2}1$. Hence,
\[
\frac{\partial_u^j\eta_k}{\eta_k}(\ell;1)\ll_{j,\varepsilon} d(\ell)^\varepsilon,
\]
and in particular we have
\[
M_{k}(\ell)=\frac{\eta_k(\ell;1)}{2^{k(k+1)/2-1}\big(k(k+1)/2\big)!|\ell_1|^{1/2}}\big(kg-d(\ell_1)\big)^{k(k+1)/2} +O_\varepsilon\big(g^{k(k+1)/2-1}d(\ell)^\varepsilon\big).
\]

For future purposes (see Section \ref{Vsquare}), we explicitly write down the main term for $k=1$:
\begin{equation}
\label{mainfirst}
M_{1}(\ell)=\frac{1}{|\ell_1|^{1/2}}\frac{1}{2\pi i}\oint_{|u|=r}\frac{\eta_1(\ell;u)du}{u^{g-d(\ell_1)+1}(1-u)^{2}}+O_\varepsilon\big(q^{-3g/2+\varepsilon g}\big)
\end{equation}
with
$$ \eta_1(\ell;u) = \prod_{P \in \mathcal{P}} \bigg(1- \frac{u^{2d(P)}}{|P|(1+|P|)} \bigg) \prod_{P | \ell}\bigg(1+\frac{1}{|P|}-\frac{u^{2d(P)}}{|P|^2}\bigg)^{-1}.$$

\subsection{Evaluate $S_{k}^{\textrm{e}}(\ell;V=\square)$}\label{Vsquare}

We proceed similarly as in [\textbf{\ref{F1}}] and [\textbf{\ref{F23}}]. First we note that as in \eqref{Cextend} we can extend the sum over $C$ in \eqref{Ske} to infinity, at the expense of an error of size $O_\varepsilon(q^{(k-4)g/2+\varepsilon g})$. So
\begin{align*}
&S_{k}^{\textrm{e}}(\ell;N;V=\square)=\frac{q}{(q-1)|\ell|}\sum_{\substack{f\in\mathcal{M}_{\leq N}\\d(f\ell)\ \textrm{even}}}\frac{\tau_k(f)}{|f|^{3/2}}\sum_{C|(f\ell)^\infty}\frac{1}{|C|^2}\\
&\qquad\qquad\qquad\bigg(q\sum_{V\in\mathcal{M}_{\leq d(f\ell)/2-g-2+d(C)}}G(V^2,\chi_{f\ell})-2\sum_{V\in\mathcal{M}_{\leq d(f\ell)/2-g-1+d(C)}}G(V^2,\chi_{f\ell})\\
&\qquad\qquad\qquad\qquad\qquad\qquad+\frac1q\sum_{V\in\mathcal{M}_{\leq d(f\ell)/2-g+d(C)}}G(V^2,\chi_{f\ell})\bigg)+O_\varepsilon\big(q^{(k-4)g/2+\varepsilon g}\big).
\end{align*}
 Applying the Perron formula in the form
\[
\sum_{n\leq N}a(n)=\frac{1}{2\pi i}\int_{|u|=r}\Big(\sum_{n=0}^{\infty}a(n)u^{n}\Big)\frac{du}{u^{N+1}(1-u)}
\] for the sums over $V$ we get
\begin{align*}
S_{k}^{\textrm{e}}(\ell;N;V=\square)& = \frac{1}{(q-1)|\ell|} \sum_{\substack{f \in \mathcal{M}_{\leq N} \\ d(f \ell) \text{ even}}} \frac{\tau_k(f) }{|f|^{3/2}} \sum_{C | (f\ell)^{\infty}}\frac{1}{|C|^2} \frac{1}{2 \pi i} \oint_{|u|=r_1}u^{-d(f)/2-d(C)}\\
& \qquad\qquad \Big( \sum_{V \in \mathcal{M}} G(V^2,\chi_{f \ell})\, u^{d(V)}\Big) \frac{(1-qu)^2du}{ u^{d(\ell)/2-g+1}(1-u) }+O_\varepsilon\big(q^{(k-4)g/2+\varepsilon g}\big),
\end{align*}
where $r_1=q^{-1-\varepsilon}$. Another application of the Perron formula, this time in the form
\[
\sum_{\substack{n\leq N\\n+l\ \textrm{even}}}a(n)=\frac{1}{2\pi i}\int_{|w|=r}\Big(\sum_{n=0}^{\infty}a(n)w^{n}\Big)\delta(l,N;w)\frac{dw}{w^{N+1}},
\] 
where
\[
\delta(l,N;w)=\frac12\bigg(\frac{1}{1-w}+\frac{(-1)^{N-l}}{1+w}\bigg)=\begin{cases}
\frac{1}{1-w^2}&N-l\ \textrm{even},\\
\frac{w}{1-w^2}&N-l\ \textrm{odd},
\end{cases}
\]for the sum over $f$ yields
\begin{align*}
&S_{k}^{\textrm{e}}(\ell;N;V=\square) = \frac{1}{(q-1)|\ell|} \\
&\qquad\quad  \frac{1}{(2 \pi i)^2} \oint_{|u|=r_1}  \oint_{|w|=r_2}\frac{\mathcal{N}_k (\ell;u,w) (1-qu)^2 dwdu }{u^{(N+d(\ell)-\mathfrak{a})/2-g+1}w^{N+1-\mathfrak{a}}(1-u) (1-uw^2)}+O_\varepsilon\big(q^{(k-4)g/2+\varepsilon g}\big),
\end{align*} where $r_2<1$,
$$ \mathcal{N}_k (\ell;u,w) = \sum_{f,V \in \mathcal{M}} \frac{\tau_k(f) G(V^2,\chi_{f \ell})}{|f|^{3/2}} \prod_{P | f \ell} \bigg(1- \frac{1}{|P|^2u^{d(P)}} \bigg)^{-1} u^{d(V)} w^{d(f)} $$
and $\mathfrak{a} \in \{0,1\}$ according to whether $N-d(\ell)$ is even or odd. 

We next write $ \mathcal{N}_k (\ell;u,w)$ as an Euler product. From Lemma \ref{L2} we have
\begin{align*}
&\sum_{f \in \mathcal{M}}   \frac{\tau_k(f) G(V^2,\chi_{f\ell})}{|f|^{3/2}} \prod_{P|f \ell} \bigg(1-\frac{1}{|P|^2 u^{d(P)}} \bigg)^{-1}w^{d(f)} \\
&\quad = \prod_{P |\ell}  \bigg(1-\frac{1}{|P|^2 u^{d(P)}} \bigg)^{-1} \prod_{P \nmid  \ell V} \bigg( 1+\frac{k w^{d(P)}}{|P|}\bigg(1-\frac{1}{|P|^2 u^{d(P)}} \bigg)^{-1}\bigg) \\
& \qquad\ \ \prod_{\substack{P\nmid \ell \\P|V}} \bigg( 1+ \sum_{j=1}^{\infty} \frac{ \tau_k(P^j) G(V^2,\chi_{P^j}) w^{j d(P)}}{|P|^{3j/2} }\bigg(1-\frac{1}{|P|^2 u^{d(P)}} \bigg)^{-1} \bigg) \\
&\qquad\quad\ \prod_{\substack{ P | \ell\\P \nmid V }} G(V^2,\chi_{P^{\ord_P(\ell)}})\prod_{\substack{ P | \ell\\P | V }} \bigg( G(V^2,\chi_{P^{\ord_P(\ell)}}) + \sum_{j=1}^{\infty} \frac{\tau_k(P^j) G(V^2,\chi_{P^{j+\ord_P(\ell)}}) w^{jd(P)}}{|P|^{3j/2}} \bigg) .
\end{align*} Note that if $P| \ell_2$ and $P \nmid V$, then the above expression is $0$. Hence we must have that $\text{rad}(\ell_2) | V$. Moreover, from the last Euler factor above, note that we must have $\ell_2 | V$, so write $V= \ell_2 V_1$. 
Using Lemma \ref{L2}, we rewrite 
$$ \prod_{\substack{P | \ell \\ P \nmid V}} G(V^2,\chi_{P^{\ord_P(\ell)}}) =\prod_{\substack{P | \ell_1 \\ P \nmid \ell_2}} |P|^{1/2} \prod_{\substack{P | \ell_1 \\ P \nmid \ell_2 \\ P |V_1}} |P|^{-1/2}.$$
By multiplicativity we then obtain
\begin{align*}
& \mathcal{N}_k  (\ell;u,w)  =  u^{d(\ell_2)}\prod_{P\in\mathcal{P}} \bigg( 1- \frac{1}{|P|^2 u^{d(P)}} \bigg)^{-1}  \\
&\ \prod_{P \nmid \ell} \bigg(1-\frac{1}{|P|^2u^{d(P)}}+\frac{k w^{d(P)}}{|P|}\\
&\qquad\qquad\qquad\qquad\qquad+ \sum_{i=1}^{\infty} u^{i d(P)} \bigg( 1-\frac{1}{|P|^2u^{d(P)}}+ \sum_{j=1}^{\infty} \frac{\tau_k(P^j) G(P^{2i},\chi_{P^j}) w^{j d(P)}}{|P|^{3j/2}}\bigg) \bigg) \\
&\ \ \prod_{P | \ell_1} \Bigg(|P|^{1/2+2 \ord_P(\ell_2)}+\sum_{i=1}^{\infty}u^{i d(P)} \sum_{j=0}^{\infty} \frac{\tau_k(P^j) G( P^{2i+ 2 \ord_P(\ell_2)} , \chi_{P^{j+1+2 \ord_P(\ell_2)}}) w^{jd(P)}}{|P|^{3j/2}} \Bigg) \\
&\ \ \ \prod_{\substack{P \nmid \ell_1   \\ P | \ell_2 }} \bigg( \varphi\big(P^{2 \ord_P(\ell_2)}\big) + \frac{k  |P|^{2 \ord_P(\ell_2)} w^{d(P)}}{|P|}\\
&\qquad\qquad\qquad\qquad\qquad\qquad+ \sum_{i=1}^{\infty} u^{i d(P)} \sum_{j=0}^{\infty} \frac{ \tau_k(P^j) G( P^{2i+ 2 \ord_P(\ell_2)} , \chi_{P^{j+2 \ord_P(\ell_2)}}) w^{jd(P)}}{|P|^{3j/2}}  \bigg).
\end{align*}

\subsubsection{The case $k=1$}
We have
\begin{align*}
\mathcal{N}_1(\ell;u,w) = \frac{|\ell|u^{d(\ell_2)}}{| \ell_1 |^{1/2}}\kappa_{1}(\ell;u,w) \mathcal{Z}(u) \mathcal{Z}\Big(\frac wq\Big)  \mathcal{Z}\Big(\frac {uw^2}{q}\Big),
\end{align*} where
\begin{equation*}
\kappa_1(\ell;u,w)=\prod_{P\in\mathcal{P}} \mathcal{D}_{1,P}(u,w)\prod_{P | \ell_1} \mathcal{H}_{1,P}(u,w)\prod_{\substack{P \nmid \ell_1 \\ P | \ell_2}} \mathcal{J}_{1,P}(u,w)
\end{equation*}
with
\begin{align*}
\mathcal{D}_{1,P}(u,w)=& \bigg(u^{d(P)}-\frac{1}{|P|^2}\bigg)^{-1}\bigg(1-\frac{w^{d(P)}}{|P|}\bigg)\\
&\qquad\qquad \bigg(u^{d(P)}+\frac{(uw)^{d(P)}(1-u^{d(P)})}{|P|}-\frac{1+(uw)^{2d(P)}}{|P|^2}+\frac{(uw^2)^{d(P)}}{|P|^3}\bigg),
\end{align*}
\begin{align*}
 \mathcal{H}_{1,P}(u,w)=& u^{d(P)}\bigg(1-u^{d(P)}+(uw)^{d(P)}-\frac{(uw)^{d(P)}}{|P|}\bigg)\\
&\qquad\qquad \bigg(u^{d(P)}+\frac{(uw)^{d(P)}(1-u^{d(P)})}{|P|}-\frac{1+(uw)^{2d(P)}}{|P|^2}+\frac{(uw^2)^{d(P)}}{|P|^3}\bigg)^{-1}
\end{align*} and
\begin{align*}
 \mathcal{J}_{1,P}(u,w)=& u^{d(P)}\bigg(1-\frac{1-w^{d(P)}+(uw)^{d(P)}}{|P|}\bigg)\\
&\qquad\qquad \bigg(u^{d(P)}+\frac{(uw)^{d(P)}(1-u^{d(P)})}{|P|}-\frac{1+(uw)^{2d(P)}}{|P|^2}+\frac{(uw^2)^{d(P)}}{|P|^3}\bigg)^{-1}.
\end{align*}
So
\begin{align*}
S_{1}^{\textrm{e}}(\ell;N;V&=\square) =\ \frac{1}{(q-1)|\ell_1|^{1/2}} \frac{1}{(2 \pi i)^2} \oint_{|u|=r_1}  \oint_{|w|=r_2}\\
&\  \frac{ \kappa_1(\ell;u,w)(1-qu) dwdu }{u^{(N+d(\ell_1)-\mathfrak{a})/2-g+1}w^{N+1-\mathfrak{a}}(1-u)(1-w) (1-uw^2)^2}+O_\varepsilon\big(q^{-3g/2+\varepsilon g}\big).
\end{align*}
Note that $ \kappa_1(1;u,w)$ is the same as $\prod_{P}\mathcal{B}_P(u,w/q)$ in [\textbf{\ref{F1}}]. 

Similarly as in [\textbf{\ref{F1}}], we take $r_1 = q^{-3/2}$ and $r_2<1$ in the double integral above. Recall from Lemma $6.3$ in [\textbf{\ref{F1}}] that 
$$ \kappa_1(1;u,w) = \mathcal{Z} \Big(  
\frac{w}{q^3u} \Big) \mathcal{Z} \Big( \frac{w^2}{q^2} \Big)^{-1}\prod_{P\in\mathcal{P}} \mathcal{R}_P(u,w),$$ where
\begin{align*}
\mathcal{R}_{P}(u,w)=& 1-\bigg(u^{d(P)}-\frac{1}{|P|^2}\bigg)^{-1}\bigg(1+\frac{w^{d(P)}}{|P|}\bigg)^{-1}\frac{w^{d(P)}}{|P|u^{d(P)}}\bigg(u^{3d(P)} + \frac{(u^3w)^{d(P)}}{|P|}\\
&\qquad\qquad - \frac{ (u^2w)^{d(P)}}{|P|^2} +\frac{(uw)^{d(P)}(1-u^{d(P)})}{|P|^3} - \frac{1+(uw)^{2d(P)}}{|P|^4} + \frac{(uw^2)^{d(P)}}{|P|^5} \bigg),
\end{align*}
and $\prod_{P\in\mathcal{P}} \mathcal{R}_P(u,w)$ converges absolutely for $|w|^2<q^3 |u|, |w| < q^4 |u|^2 , |w|<q$ and $|wu| <1$.
In the double integral, we enlarge the contour of integration over $w$ to $|w| = q^{3/4-\varepsilon}$ and encounter two poles at $w=1$ and $w=q^2u$. Let $A(\ell;N)$ be the residue of the pole at $w=1$ and $B(\ell;N)$ be the residue of the pole at $w=q^2u$. By bounding the integral on the new contour, we can write $$S_{1}^{\textrm{e}}(\ell;N;V=\square) = A(\ell;N)+B(\ell;N)+O_\varepsilon\big(|\ell_1|^{1/4}q^{-3g/2+\varepsilon g }\big).$$ 

For the residue at $w=1$ we have
$$A(\ell;N) = \frac{1}{(q-1)|\ell_1|^{1/2}} \frac{1}{2 \pi i} \oint_{|u|=r_1} \frac{\kappa_1(\ell;u,1) (1-qu)du}{u^{(N+d(\ell_1)-\mathfrak{a})/2-g+1} (1-u)^3} .$$
We make the change of variables $u \mapsto 1/u^2$ and use the fact that
$$ 1- \frac{q}{u^2} = - \frac{q}{u^2 \mathcal{Z} ( \frac{u^2}{q^2} )}\qquad \text{  and  }\qquad \Big(1-\frac{1}{u^2} \Big)^{-1} = \mathcal{Z}\Big( \frac{1}{qu^2} \Big).$$
A direct computation with the  Euler product shows that
\begin{align*} \zeta_q(2) \kappa_1\Big( \ell;\frac{1}{u^2},1\Big) \mathcal{Z} \Big(  \frac{1}{qu^2} \Big)  \mathcal{Z} \Big( \frac{u^2}{q^2} \Big)^{-1}  &= \prod_{P \in \mathcal{P}} \bigg(1- \frac{u^{2d(P)}}{|P|(1+|P|)} \bigg) \prod_{P | \ell}\bigg(1+\frac{1}{|P|}-\frac{u^{2d(P)}}{|P|^2}\bigg)^{-1}\\
& = \eta_1(\ell;u).\end{align*}
So after the change of variables, we have
$$A(\ell;N) = -\frac{1}{|\ell_1|^{1/2}} \frac{1}{2 \pi i} \oint_{|u| = r_1^{-1/2}} \frac{  \eta_1(\ell;u) du}{u^{2g-N-d(\ell_1)+\mathfrak{a}-1} (1-u^2)^2},$$
and hence
$$A(\ell):=A(\ell;g-1)+A(\ell;g) = -\frac{1}{|\ell_1|^{1/2}} \frac{1}{2 \pi i} \oint_{|u| = r_1^{-1/2}} \frac{  \eta_1(\ell;u)\big(u^{\mathfrak{a}(g)}+ u^{2-\mathfrak{a}(g)}\big)du}{u^{g-d(\ell_1)+1} (1-u^2)^2}.$$

Consider the integral
\[
\frac{1}{2\pi i}\oint_{|u| = r_1^{-1/2}}\frac{\eta_1(\ell;u)du}{u^{g-d(\ell_1)+1}(1-u)^{2}}.
\]
Making the change of variables $u \mapsto -u$ and using the facts that $\eta_1(\ell;u)$ is an even function and that $(-1)^{g-d(\ell_1)}=(-1)^{\mathfrak{a}(g)}$ (which follows from the definition of $\mathfrak{a}(g)$), this is equal to
\[
\frac{1}{2\pi i}\oint_{|u| = r_1^{-1/2}}\frac{(-1)^{\mathfrak{a}(g)}\eta_1(\ell;u)du}{u^{g-d(\ell_1)+1}(1+u)^{2}}.
\]
Hence we get
\begin{align*}
\frac{2}{2\pi i}\oint_{|u| = r_1^{-1/2}}\frac{\eta_1(\ell;u)du}{u^{g-d(\ell_1)+1}(1-u)^{2}}=\frac{1}{2\pi i}\oint_{|u| = r_1^{-1/2}}\frac{\eta_1(\ell;u)\big((1+u)^2+(-1)^{\mathfrak{a}(g)}(1-u)^2\big)du}{u^{g-d(\ell_1)+1}(1-u^2)^{2}}.
\end{align*}
Note that $(1+u)^2+(-1)^{\mathfrak{a}(g)}(1-u)^2=2\big(u^{\mathfrak{a}(g)}+ u^{2-\mathfrak{a}(g)}\big)$, and so
\[
A(\ell)=-\frac{1}{|\ell_1|^{1/2}} \frac{1}{2 \pi i}\oint_{|u| = r_1^{-1/2}}\frac{\eta_1(\ell;u)du}{u^{g-d(\ell_1)+1}(1-u)^{2}}.
\]

Now recall the expression \eqref{mainfirst} for the main term $M_1(\ell)$. Since the integrand above has no poles other than at $u=1$ between the circles of radius $r$ and $r_1^{-1/2}$ (recall that $r<1$ and $r_1^{-1/2}= q^{3/4})$, it follows that
$$A(\ell)+M_1(\ell) = - \frac{1}{| \ell_1|^{1/2}} \text{Res} (u=1)+ O_{\varepsilon}(q^{-3g/2+\varepsilon g}).$$
Note that the residue computation was done in Section \ref{main}.

Next we compute the residue at $w=q^2u$. We have
\begin{align*}
B(\ell;N) = \frac{1}{(q-1)| \ell_1|^{1/2}} \frac{1}{2 \pi i} \oint_{|u|=r_1} & \frac{(1-qu) (1-q^3u^2) }{u^{(N+d(\ell_1)-\mathfrak{a})/2-g+1}(q^2u)^{N-\mathfrak{a}}(1-u) (1-q^2u) (1-q^4u^3)^2}\\
& \prod_{P\in\mathcal{P}} \mathcal{R}_P(u,q^2u)  \prod_{P | \ell_1} \mathcal{H}_{1,P}(u,q^2u) \prod_{\substack{P \nmid \ell_1 \\ P | \ell_2}} \mathcal{J}_{1,P}(u,q^2u) \, du.
\end{align*}
We shift the contour of integration to $|u| =q^{-1-\varepsilon}$ and encounter a double pole at $u=q^{-4/3}.$ The integral over the new contour will be bounded by $q^{-3g/2+\varepsilon g}$, and after computing the residue at $u= q^{-4/3}$, it follows that
$$B(\ell;g)+B(\ell;g-1) =|\ell_1|^{1/6} q^{-4g/3}  P\big(g+d(\ell_1)\big) + O \big( q^{-3g/2 + \varepsilon g} \big),$$ where $P(x)$ is a linear polynomial whose coefficients can be written down explicitly.

%{\color{blue} I guess this part is not really necessary for our computations, since we are only interested in the leading term, but I included it anyway.}
\subsubsection{The case $k=2$}
We have
\begin{align*}
\mathcal{N}_2(\ell;u,w) &= \frac{|\ell|u^{d(\ell_2)}}{| \ell_1 |^{1/2}}\kappa_{2}(\ell;u,w)  \mathcal{Z}(u)\mathcal{Z}\Big(\frac wq\Big)^2 \mathcal{Z}\Big(\frac{uw^2}{q}\Big) \mathcal{Z} \Big( \frac{1}{q^2u} \Big) ,
\end{align*}
 where
 \begin{equation}\label{kappa2}
\kappa_2(\ell;u,w)=\prod_{P\in\mathcal{P}} \mathcal{D}_{2,P}(u,w)\prod_{P | \ell_1} \mathcal{H}_{2,P}(u,w)\prod_{\substack{P \nmid \ell_1 \\ P | \ell_2}} \mathcal{J}_{2,P}(u,w)
\end{equation}
with
\begin{align*}
 \mathcal{D}_{2,P}(u,w) =&\bigg( 1-\frac{ w^{d(P)}}{|P|}\bigg)^2\bigg(1-\frac{(uw^2)^{d(P)}}{|P|}\bigg)^{-1}\bigg(1 +\frac{w^{d(P)}\big(2-2u^{d(P)}+(uw)^{d(P)}\big)}{|P|}\\
&\qquad\quad -\frac{\big(u^{-d(P)}+3(uw^2)^{d(P)}\big)}{|P|^2}+\frac{w^{ 2d(P)}\big(2+(uw)^{2d(P)}\big)}{|P|^3}-\frac{(uw^4)^{d(P)}}{|P|^4}\bigg),
 \end{align*} 
\begin{align*}
& \mathcal{H}_{2,P}(u,w) =  \bigg(1-u^{d(P)}+2(uw)^{d(P)}-\frac{(uw)^{d(P)}\big(2-w^{d(P)}+(uw)^{d(P)}}{|P|}\bigg)\\
&\qquad\qquad \bigg(1 +\frac{w^{d(P)}\big(2-2u^{d(P)}+(uw)^{d(P)}\big)}{|P|}\\
&\qquad\qquad\qquad\qquad-\frac{\big(u^{-d(P)}+3(uw^2)^{d(P)}\big)}{|P|^2} +\frac{w^{ 2d(P)}\big(2+(uw)^{2d(P)}\big)}{|P|^3}-\frac{(uw^4)^{d(P)}}{|P|^4}\bigg)^{-1}
\end{align*}
and
\begin{align*}
& \mathcal{J}_{2,P}(u,w) = \Big(1- \frac{1-2w^{d(P)}+2(uw)^{d(P)}- (uw^2)^{d(P)}}{|P|}-\frac{(uw^2)^{d(P)}}{|P|^2}\Big)\\
&\qquad\qquad \bigg(1 +\frac{w^{d(P)}\big(2-2u^{d(P)}+(uw)^{d(P)}\big)}{|P|}\\
&\qquad\qquad\qquad\qquad -\frac{\big(u^{-d(P)}+3(uw^2)^{d(P)}\big)}{|P|^2}+\frac{w^{ 2d(P)}\big(2+(uw)^{2d(P)}\big)}{|P|^3}-\frac{(uw^4)^{d(P)}}{|P|^4}\bigg)^{-1}.\end{align*}
Hence
\begin{align}\label{S2integral}
S_{2}^{\textrm{e}}(\ell;&N;V=\square) =\ - \frac{q}{(q-1)|\ell_1|^{1/2}} \frac{1}{(2 \pi i)^2} \oint_{|u|=r_1}  \oint_{|w|=r_2}\nonumber  \\
&\qquad \frac{\kappa_2(\ell;u,w) dwdu  }{u^{(N+d(\ell_1)-\mathfrak{a})/2-g}w^{N+1-\mathfrak{a}}(1-u)(1-w)^2 (1-uw^2)^2}+O_\varepsilon\big(q^{-g+\varepsilon g}\big).
\end{align}
Note that $\kappa_2(1;u,w) $ is the same as $\mathcal{F}(u,w/q)$ in Lemma 4.3 of [\textbf{\ref{F23}}], and, hence, $\kappa_2(\ell;u,w) $ is absolutely convergent for $|u| > 1/q$, $|w| < 
q^{1/2}$, $|uw| < 1$ and $|uw^2| < 1$.

We first shift the contour $|u|=r_1$ to $|u|=r_1'=q^{-1+\varepsilon}$, and then the contour $|w|=r_2$ to $|w|=r_2'=q^{1/2-\varepsilon}$ in the expression \eqref{S2integral}. In doing so, we encounter a double pole at $w=1$. Moreover, the new integral is bounded by $O_\varepsilon(q^{-g+\varepsilon g})$. Hence
\begin{align*}
S_{2}^{\textrm{e}}(\ell,N;V=\square) =\ & \frac{q}{(q-1) |\ell_1|^{\frac{1}{2}}} \frac{1}{2 \pi i} \oint_{|u| = r_1'}\bigg( \frac{ \partial_ w \mathcal{\kappa}_2}{\kappa_2}(\ell;u,1) + \frac{5u-1}{1-u}-N+\mathfrak{a} \bigg)\\
&\qquad\qquad \frac{\kappa_2(\ell;u,1) du  }{u^{(N+d(\ell_1)-\mathfrak{a})/2-g}(1-u)^3}  + O_\varepsilon(q^{-g+\varepsilon g}),
\end{align*}
and so letting $N=2g$ and $N=2g-1$ we obtain
\begin{align}\label{S2small}
S_{2}^{\textrm{e}}(\ell;V=\square) =\ & \frac{q}{(q-1) |\ell_1|^{\frac{1}{2}}} \frac{1}{2 \pi i} \oint_{|u| = r_1'}\bigg( \frac{ \partial_ w \mathcal{\kappa}_2}{\kappa_2}(\ell;u,1)  -2g+ \frac{5u-1}{1-u}+\frac{2u}{u+u^{\mathfrak{a}(\ell)}} \bigg)\nonumber\\
&\qquad\qquad \frac{\kappa_2(\ell;u,1)(u+u^{\mathfrak{a}(\ell)}) du  }{u^{(d(\ell_1)+\mathfrak{a}(\ell))/2}(1-u)^3}  + O_\varepsilon(q^{-g+\varepsilon g}),
\end{align}
where $\mathfrak{a}(\ell) \in \{0,1\}$ according to whether $d(\ell)$ is even or odd.

It is a straightforward exercise to verify that
\begin{equation}\label{D2u}
\mathcal{D}_{2,P} (u,1)=\bigg(1-\frac{1}{|P|}\bigg)^2\bigg(1+\frac{2}{|P|}-\frac{u^{-d(P)}+u^{d(P)}}{|P|^2}+\frac{1}{|P|^3}\bigg)=\mathcal{D}_{2,P}\Big (\frac 1u,1\Big),
\end{equation}
$$ \mathcal{H}_{2,P} \Big( u, 1 \Big) =\big(1+u^{d(P)}\big)\bigg(1+\frac{2}{|P|}-\frac{u^{-d(P)}+u^{d(P)}}{|P|^2}+\frac{1}{|P|^3}\bigg)^{-1}  = u^{d(P)}\mathcal{H}_{2,P} \Big( \frac{1}{u}, 1 \Big),$$
$$ \mathcal{J}_{2,P} \Big(u, 1 \Big) =\bigg(1+\frac{1}{|P|}\bigg)\bigg(1+\frac{2}{|P|}-\frac{u^{-d(P)}+u^{d(P)}}{|P|^2}+\frac{1}{|P|^3}\bigg)^{-1}= \mathcal{J}_{2,P} \Big( \frac{1}{u},1\Big) , $$
and hence
\begin{equation}
\kappa_2(\ell;u,1)=u^{d(\ell_1)}\kappa_2\Big(\ell;\frac 1u,1\Big). \label{simf2}
\end{equation}
Let
\[
 \alpha_P(u)=\frac{\partial_ w \mathcal{D}_{2,P}}{\mathcal{D}_{2,P}}(u,1),\qquad\beta_P(u) =   \frac{ \partial_ w \mathcal{H}_{2,P}}{ \mathcal{H}_{2,P}}(u,1)\qquad\textrm{and}\qquad
\gamma_P(u) =   \frac{\partial_ w \mathcal{J}_{2,P}}{ \mathcal{J}_{2,P}}(u,1).\]

By direct computation we obtain
\begin{align*}
&\alpha_P(u)=\ 2 d(P)\bigg(1-\frac{1}{|P|}\bigg)^{-1}\bigg(1-\frac{u^{d(P)}}{|P|}\bigg)^{-1}\bigg(1+\frac{2}{|P|}-\frac{u^{-d(P)}+u^{d(P)}}{|P|^2}+\frac{1}{|P|^3}\bigg)^{-1}\\
&\qquad\bigg(\frac{u^{d(P)}}{|P|}-\frac{3+u^{d(P)}}{|P|^2}+\frac{u^{-d(P)}+1+4u^{d(P)}+u^{2d(P)}}{|P|^3}-\frac{3+u^{d(P)}+2u^{2d(P)}}{|P|^4}+\frac{2u^{d(P)}}{|P|^5}\bigg),
\end{align*}
\begin{align*}
\beta_P(u)=\ &2d(P)\big(1+u^{d(P)}\big)^{-2} \bigg(u^{d(P)}-\frac{1-u^{d(P)}}{|P|}+\frac{u^{2 d(P)}+2 u^{d(P)}-1}{|P|^2}-\frac{u^{d(P)}+2}{|P|^3}\bigg)
\end{align*}
and
\[
\gamma_P(u)=2d(P)\bigg(1+\frac{1}{|P|}\bigg)^{-2}\bigg(u^{d(P)}-\frac{1}{|P|}\bigg)\bigg(\frac{u^{-d(P)}+2}{|P|^2}+\frac{1}{|P|^3}\bigg) .
\]
From Lemma 4.4 of [\textbf{\ref{F23}}] we have
\begin{equation*}
\sum_{P\in\mathcal{P}}\alpha_P(u)=\sum_{P\in\mathcal{P}}\alpha_P\Big(\frac 1u\Big)+4u\sum_{P\in\mathcal{P}}\frac{\partial_u \mathcal{D}_{2,P}}{\mathcal{D}_{2,P}}(u,1)+\frac{2(1+u)}{1-u}.
\end{equation*}
Also note that
$$\beta_P(u)-\beta_P\Big(\frac 1u\Big) = 4 u \frac{\partial_u \mathcal{H}_{2,P} }{\mathcal{H}_{2,P}}(u, 1)-2d(P)$$
and
$$\gamma_P(u)-\gamma_P\Big(\frac1u\Big) = 4 u \frac{\partial_u \mathcal{J}_{2,P}}{\mathcal{J}_{2,P}}(u, 1).$$
Combining the above equations we get that
\begin{equation}
 \frac{ \partial_ w \mathcal{\kappa}_2}{\kappa_2}(\ell;u,1)  =  \frac{ \partial_ w \mathcal{\kappa}_2}{\kappa_2}\Big(\ell;\frac1u,1\Big)  + 4u \frac{ \partial_u \kappa_2}{\kappa_2} (\ell;u,1)+\frac{2(1+u)}{1-u}- 2 d(\ell_1).
\label{simbeta}
\end{equation}

We remark from \eqref{D2u} that $\kappa_2 (\ell;u,1)$ is analytic for $1/q<|u|<q$. Making a change of variables $u \mapsto 1/u$ in the integral \eqref{S2small} and using equations \eqref{simf2} and \eqref{simbeta} we get
\begin{align*}
&S_{2}^{\textrm{e}}(\ell;V=\square)= - \frac{q}{(q-1) |\ell_1|^{\frac{1}{2}}} \frac{1}{2 \pi i} \oint_{|u| = q^{1-\varepsilon}}\frac{\kappa_2(\ell;u,1)(u+u^{\mathfrak{a}(\ell)}) }{u^{(d(\ell_1)+\mathfrak{a}(\ell))/2}  (1-u)^3}\\
&\quad \bigg(\frac{ \partial_ w \kappa_2}{\kappa_2}(\ell;u,1)    -4u \frac{ \partial_u \kappa_2}{\kappa_2} (\ell;u,1)+ 2 d( \ell_1)-2g- \frac{7+u}{1-u}+\frac{2u^{\mathfrak{a}(\ell)}}{u+u^{\mathfrak{a}(\ell)}} \bigg) du+O_\varepsilon(q^{-g+ \varepsilon g}).
\end{align*}
We further use the fact that
\begin{align*}
&\frac{1}{2 \pi i}  \oint_{|u| = q^{1-\varepsilon}}\frac{ \partial_u \kappa_2(\ell;u,1)(u+u^{\mathfrak{a}(\ell)})du }{u^{(d(\ell_1)+\mathfrak{a}(\ell))/2-1}  (1-u)^3}\\
&\quad= - \frac{1}{2 \pi i}  \oint_{|u| = q^{1-\varepsilon}} \frac{\kappa_2(\ell;u,1)(u+u^{\mathfrak{a}(\ell)})}{u^{(d(\ell_1)+\mathfrak{a}(\ell))/2-1}(1-u)^3}\bigg(\frac{u}{u+u^{\mathfrak{a}(\ell)}} -\frac{d(\ell_1)}{2}+\frac{1+2u}{1-u}\bigg)\, du,
\end{align*}
as $\mathfrak{a}(\ell)(u^{\mathfrak{a}(\ell)}-u)=0$. Combining the two equations above, it follows that
\begin{align}\label{S2large}
S_{2}^{\textrm{e}}(\ell;V=\square)= & - \frac{q}{(q-1) |\ell_1|^{\frac{1}{2}}} \frac{1}{2 \pi i} \oint_{|u| = q^{1-\varepsilon}}\frac{\kappa_2(\ell;u,1)(u+u^{\mathfrak{a}(\ell)}) }{u^{(d(\ell_1)+\mathfrak{a}(\ell))/2}  (1-u)^3}\\
&\qquad\qquad \bigg( \frac{ \partial_ w \kappa_2}{\kappa_2}(\ell;u,1) -2g+ \frac{5u-1}{1-u}+\frac{2u}{u+u^{\mathfrak{a}(\ell)}}  \bigg) du+O_\varepsilon(q^{-g+ \varepsilon g}).\nonumber
\end{align}
As there is only one pole of the integrand at $u=1$ in the annulus between $|u|=q^{-1+\varepsilon}$ and $|u|=q^{1-\varepsilon}$, in view of \eqref{S2small} and \eqref{S2large} we conclude that
 \begin{equation*}
S_{2}^{\textrm{e}}(\ell;V=\square)  = - \frac{q }{2(q-1)|\ell_1|^{\frac{1}{2}}}\text{Res}(u=1) +O_\varepsilon(q^{-g+ \varepsilon g}).
\end{equation*}

To compute the residue at $u=1$, we proceed as in calculating the residue of $M_k(\ell)$ in the previous subsection. In doing so we have
\begin{align*}
\textrm{Res}(u=1)=&\ \kappa_2(\ell;1,1)\bigg(\frac{g}{2}\sum_{j=0}^{2}\frac{\partial_u^j\kappa_2}{\kappa_2}(\ell;1,1)Q_{2,j}\big(d(\ell_1)\big)\\
&\qquad\qquad\qquad\qquad\qquad\qquad-\frac{1}{6}\sum_{i=0}^{1}\sum_{j=0}^{3-i}\frac{\partial_u^j\partial_w^i\kappa_2}{\kappa_2}(\ell;1,1)R_{2,i,j}\big(d(\ell_1)\big)\bigg),
\end{align*}
where $Q_{2,j}(x)$'s and $R_{2,i,j}(x)$'s are explicit polynomials of degrees $2-j$ and $3-i-j$, respectively, and the leading coefficients of $Q_{2,0}(x)$ and $R_{2,0,0}(x)$ are $1$. We also note that
\[
\kappa_2(\ell;1,1)=\frac{\eta_2(\ell;1)}{\zeta_q(2)}=\frac{(q-1)\eta_2(\ell;1)}{q},
\]
and as before we have the estimates
\[
\frac{\partial_u^j\kappa_2}{\kappa_2}(\ell;1,1),\, \frac{\partial_u^j\partial_w\kappa_2}{\kappa_2}(\ell;1,1)\ll_{j,\varepsilon}d(\ell)^\varepsilon.
\]
Hence, in particular, we get
\[
S_{2}^{\textrm{e}}(\ell;V=\square)=-\frac{\eta_2(\ell;1)}{12|\ell_1|^{1/2}}\Big(3gd(\ell_1)^2-d(\ell_1)^3\Big)+O\big(gd(\ell_1)d(\ell)^\varepsilon\big).
\]

 \subsubsection{The case $k=3$}
 We have
 \begin{align*}
\mathcal{N}_3(\ell;z,w) =& \frac{|\ell| u^{d(\ell_2)}}{| \ell_1 |^{1/2}}\kappa_3(\ell;u,w)\mathcal{Z}(u) \mathcal{Z}\Big(\frac wq\Big)^3  \mathcal{Z}\Big(\frac{uw^2}{q}\Big)^6 \mathcal{Z} \Big(  \frac{1}{q^2u} \Big) \mathcal{Z}(uw/q)^{-3}  ,
\end{align*}
 where
 \begin{equation}\label{kappa3}
\kappa_3(\ell;u,w)=\prod_{P\in\mathcal{P}} \mathcal{D}_{3,P}(u,w)\prod_{P | \ell_1} \mathcal{H}_{3,P}(u,w)\prod_{\substack{P \nmid \ell_1 \\ P | \ell_2}} \mathcal{J}_{3,P}(u,w)
\end{equation}
with
\begin{align*}
&\mathcal{D}_{3,P}(u,w) =\bigg (1-\frac{w^{d(P)}}{|P|}\bigg)^3 \bigg(1\frac{(uw)^{d(P)}}{|P|}\bigg)^{-3}\bigg (1-\frac{(uw^2)^{d(P)}}{|P|}\bigg)^3 \\
&\qquad\bigg(1+\frac{3w^{d(P)}\big(1-u^{d(P)}+(uw)^{d(P)}\big)}{|P|}-\frac{u^{-d(P)}+(uw^2)^{d(P)}\big(6-w^{d(P)}+(uw)^{d(P)}\big)}{|P|^2}\\
&\qquad\qquad+\frac{3w^{2d(P)}\big(1+(uw)^{2d(P)}\big)}{|P|^3}-\frac{(uw^4)^{d(P)}\big(3+(uw)^{2d(P)}\big)}{|P|^4}+\frac{(uw^3)^{2d(P)}}{|P|^5}\bigg),
\end{align*}
$\mathcal{H}_{3,P} (u,w) =$
\begin{align*}
& \bigg(1-u^{d(P)}+3(uw)^{d(P)}-\frac{(uw)^{d(P)}\big(3-3w^{d(P)}+3(uw)^{d(P)}-(uw^2)^{d(P)}\big)}{|P|}-\frac{(u^2w^3)^{d(P)}}{|P|^2}\bigg)\\
&\qquad\bigg(1+\frac{3w^{d(P)}\big(1-u^{d(P)}+(uw)^{d(P)}\big)}{|P|}-\frac{u^{-d(P)}+(uw^2)^{d(P)}\big(6-w^{d(P)}+(uw)^{d(P)}\big)}{|P|^2}\\
&\qquad\qquad+\frac{3w^{2d(P)}\big(1+(uw)^{2d(P)}\big)}{|P|^3}-\frac{(uw^4)^{d(P)}\big(3+(uw)^{2d(P)}\big)}{|P|^4}+\frac{(uw^3)^{2d(P)}}{|P|^5}\bigg)^{-1}
\end{align*}
and $\mathcal{J}_{3,P} (u,w) =$
\begin{align*}
& \bigg(1-\frac{\big(1-3w^{d(P)}+3(uw)^{d(P)}-3(uw^2)^{d(P)}\big)}{|P|}-\frac{(uw^2)^{d(P)}\big(3-w^{d(P)}+(uw)^{d(P)}\big)}{|P|^2}\bigg)\\
&\qquad\bigg(1+\frac{3w^{d(P)}\big(1-u^{d(P)}+(uw)^{d(P)}\big)}{|P|}-\frac{u^{-d(P)}+(uw^2)^{d(P)}\big(6-w^{d(P)}+(uw)^{d(P)}\big)}{|P|^2}\\
&\qquad\qquad+\frac{3w^{2d(P)}\big(1+(uw)^{2d(P)}\big)}{|P|^3}-\frac{(uw^4)^{d(P)}\big(3+(uw)^{2d(P)}\big)}{|P|^4}+\frac{(uw^3)^{2d(P)}}{|P|^5}\bigg)^{-1}.
\end{align*}
 Using the above, we obtain
\begin{align}\label{S3integral}
&S_{3}^{\textrm{e}}(\ell;N;V=\square) =- \frac{q}{(q-1)|\ell_1|^{1/2}} \frac{1}{(2 \pi i)^2} \oint_{|u|=r_1}  \oint_{|w|=r_2} \\
&\qquad\qquad \frac{(1-uw)^3\kappa_3(\ell;u,w) dwdu }{u^{(N+d(\ell_1)-\mathfrak{a})/2-g}w^{N+1-\mathfrak{a}}(1-u)(1-w)^3 (1-uw^2)^7}+O\big(q^{-g/2+\varepsilon g}\big).\nonumber
\end{align}
Note that $\kappa_3(1;u,w)$ is the same as $\mathcal{T}(u,w/q)$ in Lemma 7.4 of [\textbf{\ref{F23}}]. As a result [\textbf{\ref{F23}}], $\kappa_3(\ell;u,w)$ is absolutely convergent for $|u| > 1/q$, $|w| < q^{1/2}$, $|uw| <q^{1/2}$ and $|uw^2| <q^{1/2}$. Moreover, $\kappa_3(\ell;u,1)$ has an analytic continuation when $1/q<|u|<q$.

We proceed as in the case $k=2$. First we move the contour $|u|=r_1$ to $|u|=r_1'=q^{-1+\varepsilon}$, and then the contour $|w|=r_2$ to $|w|=r_2'=q^{1/2-\varepsilon}$ in the equation \eqref{S3integral}. In doing so, we cross a triple pole at $w=1$. On the new contours, the integral is bounded by $O_\varepsilon(q^{-g+\varepsilon g})$. Hence, by expanding the terms in their Laurent series,
\begin{align*}
S_{3}^{\textrm{e}}(\ell,N;V=\square) &=\ \frac{q}{(q-1) |\ell_1|^{\frac{1}{2}}} \frac{1}{2 \pi i} \oint_{|u| = r_1'} \frac{\kappa_3(\ell;u,1)}{u^{(N+d(\ell_1)-\mathfrak{a})/2-g}(1-u)^7}\nonumber \\
&\quad\sum_{i_1=0}^{2}\sum_{i_2=0}^{2-i_1}\frac{\partial_w^{i_1}\kappa_3}{\kappa_3}(\ell;u,1)Q_{3,i_1,i_2}(\mathfrak{a},u)(1-u)^{i_1+i_2}N^{i_2} du + O_\varepsilon(q^{-g/2+\varepsilon g}),
\end{align*}
where $Q_{3,i,j}(\mathfrak{a},u)$'s are some explicit functions and are analytic as functions of $u$.

Next we move the $u$-contour to $|u|=q^{1-\varepsilon}$. We encounter a pole at $u=1$ and we bound the new integral by $O_\varepsilon(|\ell_1|^{-1}q^{-g/2+\varepsilon g})$. For the residue at $u=1$, we calculate the Taylor series of the terms in the integrand and get
\begin{align*}
\textrm{Res}(u=1)=&\kappa_3(\ell;1,1)\sum_{i_1=0}^{2}\sum_{i_2=0}^{2-i_1}\sum_{j=0}^{6-i_1-i_2}\frac{\partial_u^j\partial_w^{i_1}\kappa_3}{\kappa_3}(\ell;1,1)R_{3,i_1,i_2,j}(\mathfrak{a},g+d(\ell_1))N^{i_2},
\end{align*}
where $R_{3,i_1,i_2,j}(\mathfrak{a},x)$'s are explicit polynomials in $x$ with degree $6-i_1-i_2-j$. Thus,
\begin{align*}
S_{3}^{\textrm{e}}(\ell;V=\square)=&\frac{\kappa_3(\ell;1,1)q}{(q-1) |\ell_1|^{1/2}}\sum_{N=3g-1}^{3g}\sum_{i_1=0}^{2}\sum_{i_2=0}^{2-i_1}\sum_{j=0}^{6-i_1-i_2}\frac{\partial_u^j\partial_w^{i_1}\kappa_3}{\kappa_3}(\ell;1,1)R_{3,i_1,i_2,j}(\mathfrak{a},g+d(\ell_1))N^{i_2}\\
&\qquad\qquad+ O_\varepsilon(q^{-g/2+\varepsilon g})+O_\varepsilon(|\ell_1|^{-3/4}q^{-g/4+\varepsilon g}).
\end{align*}

As for the leading term, as before we can show that
\[
\kappa_3(\ell;1,1)=\frac{\eta_3(\ell;1)}{\zeta_q(2)} ,\qquad\frac{\partial_u^j\partial_w^{i_1}\kappa_3}{\kappa_3}(\ell;1,1)\ll_{i_1,j,\varepsilon}d(\ell)^\varepsilon,
\]
and so
\begin{align*}
S_{3}^{\textrm{e}}(\ell;V=\square)&=\frac{\eta_3(\ell;1)}{ |\ell_1|^{1/2}}\sum_{N=3g-1}^{3g}\sum_{i_2=0}^{2}R_{3,0,i_2,0}(\mathfrak{a},g+d)N^{i_2}+O_\varepsilon\big(g^5d(\ell)^\varepsilon\big)\\
&= \frac{\eta_3(\ell;1)}{ |\ell_1|^{\frac{1}{2}}}\sum_{N=3g-1}^{3g}\sum_{i_2=0}^{2} \frac{1}{2 \pi i} \oint_{|u| = r_1'} \frac{Q_{3,0,i_2}(\mathfrak{a},1)(1-u)^{i_2}N^{i_2} du}{u^{(g+d(\ell_1))/2}(1-u)^7}\nonumber \\
&\qquad\qquad  +O_\varepsilon\big(g^5d(\ell)^\varepsilon\big).
\end{align*}
Expanding the terms in \eqref{S3integral} in the Laurent series
\begin{align*}
&(1-uw)^3=(1-u)^3-3u(1-u)^2(w-1)+3u^2(1-u)(w-1)^2\ldots,\\
&w^{-N}=1-N(w-1)+\frac{N(N+1)}{2}(w-1)^2\ldots,\\
&(1-uw^2)^{-7}=(1-u)^{-7}+14u(1-u)^{-8}(w-1)+14u(1-u)^{-9}(1+7u)(w-1)^2\ldots,
\end{align*}
we see that
\begin{align*}
S_{3}^{\textrm{e}}(\ell;V=\square)&= \frac{\eta_3(\ell;1)}{ |\ell_1|^{\frac{1}{2}}}\sum_{N=3g-1}^{3g} \frac{1}{2 \pi i} \oint_{|u| = r_1'} \frac{1}{u^{(g+d(\ell_1))/2}(1-u)^7}\nonumber \\
&\qquad\qquad  \bigg(73-11(1-u)N+\frac{(1-u)^2N^2}{2}\bigg)du+O_\varepsilon\big(g^5d(\ell)^\varepsilon\big)\\
&=-\frac{\eta_3(\ell;1)}{2^56! |\ell_1|^{\frac{1}{2}}}(g+d)^4\Big(73(g+d)^2-396g(g+d)+540g^2\Big)+O_\varepsilon\big(g^5d(\ell)^\varepsilon\big).
\end{align*}

\subsection{Bounding $S_k(\ell;N;V \neq \square)$}\label{Vnonsquare}
Recall from \eqref{Sknonsquare} that 
$$ S_k(\ell;N;V \neq \square) = S_{k}^{\textrm{e}}(\ell;N;V\ne\square) +S_{k}^{\textrm{o}}(\ell;N) , $$ with $S_{k}^{\textrm{e}}(\ell;N;V\ne\square) = S_{k,1}^{\textrm{e}}(\ell;N;V\ne\square)-qS_{k,2}^{\textrm{e}}(\ell;N;V\ne\square)$ and $S_{k}^{\textrm{o}}(\ell;N) = S_{k,1}^{\textrm{o}}(\ell;N)-qS_{k,2}^{\textrm{o}}(\ell;N)$, and $S_{k,1}^{\textrm{o}}(\ell;N)$ is given by equation \eqref{s1odd}. We will focus on bounding $S_{k,1}^{\textrm{o}}(\ell;N)$, since bounding the other ones follow similarly.

Using the fact that for $r_1<1$,
$$\sum_{\substack{C \in \mathcal{M}_j \\ C| (f \ell)^{\infty}}} \frac{1}{|C|^2} = \frac{1}{2 \pi i} \oint_{|u|=r_1} q^{-2j}   \prod_{P | f \ell} \big(1-u^{d(P)}\big)^{-1}\, \frac{du}{u^{j+1}},$$ and writing $V=V_1V_2^2$ with $V_1$ a square-free polynomial, we have
\begin{align*}
S_{k,1}^{\textrm{o}}(\ell;N) &= \frac{q^{3/2}}{(q-1)|\ell|} \frac{1}{2 \pi i} \oint_{|u|=r_1} \sum_{\substack{n\leq N \\ n +d(\ell) \text{ odd}}}\sum_{j=0}^g  \sum_{\substack{r\leq n+d(\ell)-2g+2j-2 \\ r \text{ odd}}}q^{-2j} \\
& \qquad \sum_{V_1 \in \mathcal{H}_r} \sum_{V_2 \in \mathcal{M}_{(n+d(\ell)-r)/2-g+j-1}} \sum_{f \in \mathcal{M}_n} \frac{\tau_k(f) G(V_1V_2^2, \chi_{f \ell})}{|f|^{3/2} }\prod_{P | f \ell} \big(1-u^{d(P)}\big)^{-1} \, \frac{du}{u^{j+1}}.
\end{align*}
Now
\begin{align*}
\sum_{f \in \mathcal{M}}   \frac{\tau_k(f) G(V_1V_2^2, \chi_{f \ell})}{|f|^{3/2}}  \prod_{P | f \ell} \big(1-u^{d(P)}\big)^{-1}w^{d(f)}=  \mathcal{H}(V,\ell;u,w) \mathcal{J}(V,\ell;u,w)\mathcal{K}(V_1 ; u,w), 
\end{align*} where
$$ \mathcal{H}(V,\ell;u,w) = \prod_{P | \ell} \bigg( \sum_{j=0}^{\infty} \frac{\tau_k(P^j) G(V,\chi_{P^{j+\text{ord}_P(\ell)}}) w^{j d(P)}}{|P|^{3j/2}}   \bigg) \big(1-u^{d(P)}\big)^{-1},$$
\begin{align*}
\mathcal{J}(V,\ell;u,w) &=  \prod_{\substack{P \nmid \ell\\P |V  }} \bigg( 1+ \sum_{j=1}^{\infty} \frac{\tau_k(P^j) G(V, \chi_{P^j}) w^{j d(P)}}{|P|^{3j/2}} \big(1-u^{d(P)}\big)^{-1}  \bigg)   \\
& \qquad\qquad \prod_{\substack{P \nmid V_1\\P | \ell V_2 }}\Big( 1+ \frac{k \chi_{V_1}(P) w^{d(P)}}{|P|} \big(1-u^{d(P)}\big)^{-1}\Big)^{-1}
\end{align*}
 and
$$ \mathcal{K}(V_1;u,w) =  \prod_{P \nmid V_1} \bigg( 1+ \frac{k \chi_{V_1}(P) w^{d(P)}}{|P|}\big(1-u^{d(P)}\big)^{-1} \bigg). $$ 
We use the Perron formula for the sum over $f$ and obtain
\begin{align*}
S_{k,1}^{\textrm{o}}(\ell;N)  &= \frac{q^{3/2}}{(q-1)|\ell|} \frac{1}{(2 \pi i)^2} \oint_{|u|=q^{-\varepsilon}} \oint_{|w| = q^{1/2-\varepsilon}} \sum_{\substack{n\leq N \\ n +d(\ell) \text{ odd}}}  \sum_{j=0}^g  \sum_{\substack{r\leq n+d(\ell)-2g+2j-2 \\ r \text{ odd}}} q^{-2j}\\
& \qquad \sum_{V_1 \in \mathcal{H}_r} \sum_{V_2 \in \mathcal{M}_{(n+d(\ell)-r)/2-g+j-1}}  \mathcal{H}(V,\ell;u,w) \mathcal{J}(V,\ell;u,w)\mathcal{K}(V_1;u,w) \,  \frac{du}{u^{j+1}} \, \frac{dw}{w^{n+1}}.
\end{align*}

 Let $j_0$ be minimal such that $|wu^{j_0}| < 1$. Then we write
\begin{equation}
\mathcal{K}(V_1;u,w) =  \mathcal{L}\Big(\frac wq,\chi_{V_1}\Big)^k \mathcal{L}\Big(\frac{uw}{q},\chi_{V_1}\Big)^k \cdot \ldots \cdot \mathcal{L}\Big(\frac{u^{j_0-1}w}{q}, \chi_{V_1}\Big)^k \mathcal{T}(V_1;u,w), \label{expr}
\end{equation} where $\mathcal{T}(V_1;u,w)$ is absolutely convergent in the selected region. We also have $$ \mathcal{J}(V,\ell;u,w)  \ll 1$$ and similarly as in the proof of Lemma $5.3$ in [\textbf{\ref{S}}],
$$ \mathcal{H} (V,\ell;u,w)  \ll_\varepsilon  |\ell|^{1/2+\varepsilon} \big| (\ell, V_2^2)\big|^{1/2}|V|^{\varepsilon}.$$
We trivially bound the sum over $V_2$. Then we use \eqref{expr} and upper bounds for moments of $L$--functions (see Theorem $2.7$ in [\textbf{\ref{F4}}]) to get that
$$ \sum_{V_1 \in \mathcal{H}_r} \bigg| \mathcal{L}\Big(\frac wq,\chi_{V_1}\Big) \mathcal{L}\Big(\frac{uw}{q},\chi_{V_1}\Big) \cdot \ldots \cdot \mathcal{L}\Big(\frac{u^{j_0-1}w}{q}, \chi_{V_1}\Big) \bigg|^k \ll_\varepsilon q^{r} r^{k(k+1)/2+\varepsilon} .$$
Alternatively, one can use a Lindel\"{o}f type bound for each $L$--function to get the weaker upper bound of $q^{r+\varepsilon r}$ for the expression above. 
Trivially bounding the rest of the expression, we obtain that
$$ S_{k,1}^{\textrm{o}}(\ell;N) \ll_\varepsilon |\ell|^{1/2} q^{N/2 -2g+ \varepsilon g}.$$ 
Hence $$S_k(\ell;N;V \neq \square) \ll_\varepsilon |\ell|^{1/2}q^{(k-4)g/2+ \varepsilon g}.$$
%\begin{align*}
%& \sum_{f \in \mathcal{M}}  w^{d(f)} \frac{\tau_k(f) G(V_1V_2^2, \chi_{f \ell})}{|f|^{\frac{1}{2}} \prod_{P | f \ell} (1-u^{d(P)})} = \prod_{P | \ell} (1-u^{d(P)})^{-1} \prod_{P \nmid V_1} \Big( 1+ \frac{k \chi_{V_1}(P) w^{d(P)}}{1-u^{d(P)}} \Big)\\
%& \times  \prod_{\substack{P | V_2 \\ P \nmid V_1}}\Big( 1+ \frac{k \chi_{V_1}(P) w^{d(P)}}{1-u^{d(P)}} \Big)^{-1} \prod_{\substack{P | \ell \\ P \nmid V}} \Big( 1+ \frac{k \chi_{V_1}(P) w^{d(P)}}{1-u^{d(P)}} \Big)^{-1} \\
%& \times \prod_{\substack{P |V \\  P \nmid \ell}} \Big( 1+ \frac{1}{1-u^{d(P)}} \sum_{j=1}^{\infty} \frac{\tau_k(P^j) G(V, \chi_{P^j}) w^{j d(P)}}{|P|^{j/2}}   \Big) \prod_{P | \ell} \Big( \sum_{j=0}^{\infty} \frac{\tau_k(P^j) G(V,\chi_{P^{\text{ord}_P(\ell)+j}}) w^{j d(P)}}{|P|^{j/2}}   \Big).
%\end{align*}

\section{Moments of the partial Hadamard product}\label{momentsZ}

\subsection{Random matrix theory model}

Recall that
\begin{equation*}
Z_{X}(s,\chi_D)=\exp\Big(-\sum_{\rho}U\big((s-\rho)\ X\big)\Big),
\end{equation*}
where
\begin{equation*}
U(z)=\int_{0}^{\infty}u(x)E_{1}(z\log x)dx.
\end{equation*}
Denote the zeros by $\rho=1/2+i\gamma$. Since $E_1(-ix)+E_1(ix)=-2\textrm{Ci}(|x|)$ for $x\in\mathbb{R}$, where $\textrm{Ci}(z)$ is the cosine integral,
\[
\textrm{Ci}(z)=-\int_{z}^{\infty}\frac{\cos(x)}{x}dx,
\]
we have
\begin{equation}\label{rmt}
\Big\langle Z_X(\chi_D)^k \Big\rangle_{\mathcal{H}_{2g+1}}=\bigg\langle \prod_{\gamma>0}\exp\Big(2k\int_{0}^{\infty}u(x)\textrm{Ci}\big(\gamma X(\log x) \big)dx\Big) \bigg\rangle_{\mathcal{H}_{2g+1}}.
\end{equation}

We model the right hand side of \eqref{rmt} by replacing the ordinates $\gamma$ by the eigenangles of a $2g\times 2g$ symplectic unitary matrix and averaging over all such matrices with respect to the Haar measure. The $k$-moment of $Z_X(\chi_D)$ is thus expected to be asymptotic to
\[
\mathbb{E}_{2g}\bigg[\prod_{n=1}^{g}\exp\Big(2k\int_{0}^{\infty}u(x)\textrm{Ci}\big(\theta_nX(\log x) \big)dx\Big)\bigg],
\]
where $\pm\,\theta_n$ with $0\leq\theta_1\leq\ldots\leq\theta_{g}\leq\pi$ are the $2g$ eigenangles of the random matrix and $\mathbb{E}_{2g}[.]$ denotes the expectation with respect to the Haar measure. It is convenient to have our function periodic, so we instead consider 
\begin{equation}\label{rmt1}
\mathbb{E}_{2g}\bigg[\prod_{n=1}^{g}\phi(\theta_n)\bigg],
\end{equation}
where
\begin{align*}
\phi(\theta)&=\exp\bigg(2k\int_{0}^{\infty}u(x)\Big(\sum_{j=-\infty}^{\infty}\textrm{Ci}\big(|\theta+2\pi j|X(\log x) \big)\Big)dx\bigg)\\
&=\Big|2\sin\frac \theta 2\Big|^{2k}\exp\bigg(2k\int_{0}^{\infty}u(x)\Big(\sum_{j=-\infty}^{\infty}\textrm{Ci}\big(|\theta+2\pi j|X(\log x) \big)\Big)dx-2k\log\Big|2\sin\frac \theta 2\Big| \bigg).
\end{align*}
The average \eqref{rmt1} over the symplectic group has been asymptotically evaluated in [\textbf{\ref{DIK}}] and we have
\[
\mathbb{E}_{2g}\bigg[\prod_{n=1}^{g}\phi(\theta_n)\bigg]\sim \frac{G(k+1)\sqrt{\Gamma(k+1)}}{\sqrt{G(2k+1)\Gamma(2k+1)}}\Big(\frac{2g}{e^\gamma X}\Big)^{k(k+1)/2}.
\]

\subsection{Proof of Theorem \ref{k123}}

What most important to us in evaluating the moments of $Z_X(\chi_D)$ is the leading term coming from the twisted moments. Theorem \ref{tfm}, Theorem \ref{tsm}, Theorem \ref{ttm} and \eqref{eta} show that
\begin{align*}
\Big\langle L(\tfrac12,\chi_D)^k\chi_D(\ell) \Big\rangle_{\mathcal{H}_{2g+1}}=&\ \frac{c_k\mathcal{A}_k\mathcal{B}_k(\ell_1)\mathcal{C}_k(\ell_1,\ell_2)}{2^{k(k+1)/2-1}(k(k+1)/2)!|\ell_1|^{1/2}}\big(kg\big)^{k(k+1)/2}\\
&\qquad\qquad+O\Big(\frac{\tau_k(\ell_1)}{|\ell_1|^{1/2}}g^{k(k+1)/2-1}d(\ell)\Big)+O_\varepsilon\big(|\ell|^{1/2}q^{(k-4)g/2+\varepsilon g}\big)
\end{align*}
with $c_1=c_2=1$ and $c_3=2^9/3^6$, where 
\begin{align*}
&\mathcal{A}_{k}=\prod_P\bigg(1-\frac{1}{|P|}\bigg)^{k(k+1)/2}\bigg(1+\bigg(1+\frac{1}{|P|}\bigg)^{-1}\sum_{j=1}^{\infty}\frac{\tau_{k}(P^{2j})}{|P|^j}\bigg),\\
&\mathcal{B}_{k}(\ell_1)=\prod_{P|\ell_1}\bigg(\sum_{j=0}^{\infty}\frac{\tau_k(P^{2j+1})}{|P|^j}\bigg)\bigg(1+\frac{1}{|P|}+\sum_{j=1}^{\infty}\frac{\tau_{k}(P^{2j})}{|P|^j}\bigg)^{-1},\\
&\mathcal{C}_{k}(\ell_1,\ell_2)=\prod_{\substack{P\nmid \ell_1\\P|\ell_2}}\bigg(\sum_{j=0}^{\infty}\frac{\tau_k(P^{2j})}{|P|^j}\bigg)\bigg(1+\frac{1}{|P|}+\sum_{j=1}^{\infty}\frac{\tau_{k}(P^{2j})}{|P|^j}\bigg)^{-1}.
\end{align*}
Combining with \eqref{P*} we get
\begin{align*}
\Big\langle L(\tfrac12,\chi_D)^kP_{-k,X}^{*}(\chi_{D}) \Big\rangle_{\mathcal{H}_{2g+1}}&=\, J_{k,1}+J_{k,2}+O_\varepsilon\big(q^{\vartheta g+(k-4)g/2+\varepsilon g}\big)+O_\varepsilon(q^{-c\vartheta g/4+\varepsilon g})
\end{align*}
for any $\vartheta>0$, where
\begin{equation}\label{Jk1}
J_{k,1}=\frac{2c_k\mathcal{A}_k}{(k(k+1)/2)!}\Big(\frac{kg}{2}\Big)^{k(k+1)/2}\sum_{\substack{\ell_1,\ell_2\in S(X)\\ \ell_1\ \textrm{square-free}\\ d(\ell_1)+2d(\ell_2)\leq\vartheta g}}\frac{\mathcal{B}_k(\ell_1)\mathcal{C}_k(\ell_1,\ell_2)\alpha_{-k}(\ell_1\ell_2^2)}{|\ell_1||\ell_2|}
\end{equation}
and
\begin{align}\label{Jk2}
J_{k,2}&\ll  g^{k(k+1)/2-1}\sum_{\ell_1,\ell_2\in S(X)}\frac{\tau_k(\ell_1)\tau_{k}(\ell_1\ell_2^2)\big(d(\ell_1)+2d(\ell_2)\big)}{|\ell_1||\ell_2|}\nonumber\\
&\ll g^{k(k+1)/2-1}\sum_{\ell_2\in S(X)}\frac{\tau_{k}(\ell_2^2)d(\ell_2)}{|\ell_2|}\sum_{\ell_1\in S(X)}\frac{\tau_k(\ell_1)^2d(\ell_1)}{|\ell_1|},
\end{align}
as $\alpha_{-k}(\ell_1\ell_2^2)\ll\tau_k(\ell_1\ell_2^2)\ll\tau_k(\ell_1)\tau_{k}(\ell_2^2)$.

We first consider the error term $J_{k,2}$. Let
\begin{displaymath}
F(\sigma)=\sum_{\ell\in S(X)}\frac{\tau_k(\ell)^2}{|\ell|^\sigma}=\prod_{d(P)\leq X}\bigg(\sum_{j=0}^{\infty}\frac{\tau_k(P^j)^2}{|P|^{j\sigma}}\bigg).
\end{displaymath}
Then $F(1)\asymp\prod_{d(P)\leq X}(1-1/|P|)^{-k^2}\asymp X^{k^2}$. We note that the sum over $\ell_1$ in \eqref{Jk2} is
\[
-\frac{F'(1)}{\log q}=F(1)\sum_{d(P)\leq X}\frac{\sum_{j=0}^{\infty}j\tau_k(P^j)^2d(P)/|P|^{j}}{\sum_{j=0}^{\infty}\tau_k(P^j)^2/|P|^{j}},
\]
and hence it is
\[
\ll X^{k^2}\sum_{d(P)\leq X}\frac{d(P)}{|P|}\ll X^{k^2+1}.
\]
Similarly we have $$\sum_{\ell_2\in S(X)}\frac{\tau_{k}(\ell_2^2)d(\ell_2)}{|\ell_2|}\ll X^{k(k+1)/2+1},$$ and hence $$J_{k,2}\ll g^{k(k+1)/2-1}X^{k(3k+1)/2+2}.$$

For the main term $J_{k,1}$, recall that the function $\alpha_{-k}(\ell)$ is given by
\begin{equation}\label{alpha}
\sum_{\ell\in\mathcal{M}}\frac{\alpha_{-k}(\ell)\chi_{D}(\ell)}{|\ell|^s}=\prod_{d(P)\leq X/2}\bigg( 1-\frac{\chi_D(P)}{|P|^s}\bigg) ^{k}\prod_{X/2<d(P)\leq X}\bigg(1-\frac{k\chi_{D}(P)}{|P|^s}+\frac{k^2\chi_{D}(P)^2}{2|P|^{2s}}\bigg) .
\end{equation}
So for $1\leq k\leq 3$, the function $\alpha_{-k}(\ell)$ is supported on quad-free polynomials. As a result, if we let
\[
\mathcal{P}_X=\prod_{d(P)\leq X}P,
\]
then for the sum over $\ell_1,\ell_2$ in \eqref{Jk1} we can write $\ell_1=\ell_1'\ell_3$ and $\ell_2=\ell_2'\ell_3$, where $\ell_1',\ell_2',\ell_3$ are all square-free, i.e. $\ell_1',\ell_2',\ell_3|\mathcal{P}_X$, and $\ell_1',\ell_2',\ell_3$ are pairwise co-prime. Hence 
\begin{align*}
J_{k,1}=&\frac{2c_k\mathcal{A}_k}{(k(k+1)/2)!}\Big(\frac{kg}{2}\Big)^{k(k+1)/2}\sum_{\substack{\ell_3|\mathcal{P}_X\\ d(\ell_3)\leq\vartheta g/3}}\frac{\mathcal{B}_k(\ell_3)\alpha_{-k}(\ell_3^3)}{|\ell_3|^2}\\
&\qquad\qquad\sum_{\substack{\ell_2|(\mathcal{P}_X/\ell_3)\\ d(\ell_2)\leq(\vartheta g-3d(\ell_3))/2}}\frac{\mathcal{C}_k(\ell_2)\alpha_{-k}(\ell_2^2)}{|\ell_2|}\sum_{\substack{\ell_1|(\mathcal{P}_X/\ell_2\ell_3)\\ d(\ell_1)\leq\vartheta g-2d(\ell_2)-3d(\ell_3)}}\frac{\mathcal{B}_k(\ell_1)\alpha_{-k}(\ell_1)}{|\ell_1|},
\end{align*}
where
\[
\mathcal{C}_k(\ell_2)=\mathcal{C}_k(1,\ell_2)=\prod_{P|\ell_2}\bigg(\sum_{j=0}^{\infty}\frac{\tau_k(P^{2j})}{|P|^j}\bigg)\bigg(1+\frac{1}{|P|}+\sum_{j=1}^{\infty}\frac{\tau_{k}(P^{2j})}{|P|^j}\bigg)^{-1}.
\]

Like in \eqref{truncation}, we can remove the condition $d(\ell_1)+2d(\ell_2)+3d(\ell_3)\leq\vartheta g$ at the cost of an error of size $O_\varepsilon\big(q^{-\vartheta g/2+\varepsilon g}\big)$. Define the following multiplicative functions
\begin{align*}
T_1(f)=\sum_{\ell|f}\frac{\mathcal{B}_k(\ell)\alpha_{-k}(\ell)}{|\ell|},\qquad T_2(f)=\sum_{\ell|f}\frac{\mathcal{C}_k(\ell)\alpha_{-k}(\ell^2)}{|\ell|T_1(\ell)}
\end{align*}
and
\[
T_3(f)=\sum_{\ell|f}\frac{\mathcal{B}_k(\ell)\alpha_{-k}(\ell^3)}{|\ell|^2T_1(\ell)T_2(\ell)}.
\]
Then
\begin{align*}
J_{k,1}=&\ \frac{2c_k\mathcal{A}_k}{(k(k+1)/2)!}\Big(\frac{kg}{2}\Big)^{k(k+1)/2}T_1\big(\mathcal{P}_X\big)T_2\big(\mathcal{P}_X\big)T_3\big(\mathcal{P}_X\big)\\
=&\ \frac{2c_k\mathcal{A}_k}{(k(k+1)/2)!}\Big(\frac{kg}{2}\Big)^{k(k+1)/2}\\
&\qquad\prod_{d(P)\leq X}\bigg(1+\frac{\mathcal{B}_k(P)\alpha_{-k}(P)}{|P|}+\frac{\mathcal{C}_k(P)\alpha_{-k}(P^2)}{|P|}+\frac{\mathcal{B}_k(P)\alpha_{-k}(P^3)}{|P|^2}\bigg).
\end{align*}
We remark that
\[
\mathcal{A}_k=\Big(1+O\big(q^{-X}/X\big)\Big)\prod_{d(P)\leq X}\bigg(1-\frac{1}{|P|}\bigg)^{k(k+1)/2}\bigg(1+\frac{1}{|P|}\bigg)^{-1}\mathcal{A}_{k}(P),
\]
where
\[
\mathcal{A}_{k}(P)=1+\frac{1}{|P|}+\sum_{j=1}^{\infty}\frac{\tau_{k}(P^{2j})}{|P|^j}.
\]
So
\begin{align*}
&J_{k,1}=\Big(1+O\big(q^{-X}/X\big)\Big) \frac{2c_k}{(k(k+1)/2)!}\Big(\frac{kg}{2}\Big)^{k(k+1)/2}\prod_{d(P)\leq X}\bigg(1-\frac{1}{|P|}\bigg)^{k(k+1)/2}\bigg(1+\frac{1}{|P|}\bigg)^{-1}\\
&\ \prod_{d(P)\leq X}\bigg(\mathcal{A}_{k}(P)+\frac{\mathcal{A}_{k}(P)\mathcal{B}_k(P)\alpha_{-k}(P)}{|P|}+\frac{\mathcal{A}_{k}(P)\mathcal{C}_k(P)\alpha_{-k}(P^2)}{|P|}+\frac{\mathcal{A}_{k}(P)\mathcal{B}_k(P)\alpha_{-k}(P^3)}{|P|^2}\bigg).
\end{align*}
We note from \eqref{alpha} that $\alpha_{-k}(P)=-k$. Also if $P\in\mathcal{P}$ with $d(P)\leq X/2$, then 
$$\alpha_{-k}(P^2)=\frac{k(k-1)}{2}\qquad \textrm{and}\qquad\alpha_{-k}(P^3)=-\frac{k(k-1)(k-2)}{6},$$ and if $P\in\mathcal{P}$ with $X/2<d(P)\leq X$, then
$$\alpha_{-k}(P^2)=\frac{k^2}{2}\qquad \textrm{and}\qquad\alpha_{-k}(P^3)=0.$$ Standard calculations also give
\[
\mathcal{A}_{k}(P)=\begin{cases}
\big(1-\frac{1}{|P|}\big)^{-1}\big(1+\frac{1}{|P|}-\frac{1}{|P|^2}\big)& k=1,\\
\big(1-\frac{1}{|P|}\big)^{-2}\big(1+\frac{2}{|P|}-\frac{2}{|P|^2}+\frac{1}{|P|^3}\big) & k=2,\\
\big(1-\frac{1}{|P|}\big)^{-3}\big(1+\frac{4}{|P|}-\frac{3}{|P|^2}+\frac{3}{|P|^3}-\frac{1}{|P|^4}\big) & k=3,
\end{cases}
\]
\[
\mathcal{A}_{k}(P)\mathcal{B}_k(P)=\begin{cases}
\big(1-\frac{1}{|P|}\big)^{-1} & k=1,\\
2\big(1-\frac{1}{|P|}\big)^{-2} & k=2,\\
\big(1-\frac{1}{|P|}\big)^{-3}\big(3+\frac{1}{|P|}\big) & k=3
\end{cases}
\]
and
\[
\mathcal{A}_{k}(P)\mathcal{C
}_k(P)=\begin{cases}
\big(1-\frac{1}{|P|}\big)^{-1} & k=1,\\
\big(1-\frac{1}{|P|}\big)^{-2}\big(1+\frac{1}{|P|}\big) & k=2,\\
\big(1-\frac{1}{|P|}\big)^{-3}\big(1+\frac{3}{|P|}\big) & k=3.
\end{cases}
\]
Hence, using Lemma \ref{mertens} we get
\begin{align*}
J_{1,1}&=\Big(1+O\big(q^{-X}/X\big)\Big)g\prod_{d(P)\leq X}\bigg(1+\frac{1}{|P|}\bigg)^{-1}\\
&\qquad\qquad\prod_{d(P)\leq X/2}\bigg(1-\frac{1}{|P|^2}\bigg)\prod_{X/2<d(P)\leq X}\bigg(1+\frac{1}{2|P|}-\frac{1}{|P|^2}\bigg)\\
&=\Big(1+O\big(q^{-X/2}/X\big)\Big) g\prod_{d(P)\leq X/2}\bigg(1-\frac{1}{|P|}\bigg)\prod_{X/2<d(P)\leq X}\bigg(1-\frac{1}{|P|}\bigg)^{1/2}\\
&= \frac{1}{\sqrt{2}}\frac{g}{e^\gamma X/2}+O\big(gX^{-2}\big)
\end{align*}
and
\begin{align*}
J_{2,1}&=\Big(1+O\big(q^{-X}/X\big)\Big)\frac{g^3}{3}\prod_{d(P)\leq X}\bigg(1-\frac{1}{|P|}\bigg)\bigg(1+\frac{1}{|P|}\bigg)^{-1}\\
&\qquad\qquad\prod_{d(P)\leq X/2}\bigg(1-\frac{1}{|P|}-\frac{1}{|P|^2}+\frac{1}{|P|^3}\bigg)\prod_{X/2<d(P)\leq X}\bigg(1+\frac{1}{|P|^3}\bigg)\\
&=\Big(1+O\big(q^{-X/2}/X\big)\Big) \frac{g^3}{3}\prod_{d(P)\leq X/2}\bigg(1-\frac{1}{|P|}\bigg)^3\prod_{X/2<d(P)\leq X}\bigg(1-\frac{1}{|P|}\bigg)^{2}\\
&=\frac{1}{12} \Big(\frac{g}{e^\gamma X/2}\Big)^3+O\big(g^3X^{-4}\big).
\end{align*}
Lastly,
\begin{align*}
J_{3,1}&=\Big(1+O\big(q^{-X}/X\big)\Big)\frac{g^6}{45}\prod_{d(P)\leq X}\bigg(1-\frac{1}{|P|}\bigg)^3\bigg(1+\frac{1}{|P|}\bigg)^{-1}\\
&\qquad\qquad\prod_{d(P)\leq X/2}\bigg(1-\frac{2}{|P|}+\frac{2}{|P|^3}-\frac{1}{|P|^4}\bigg)\prod_{X/2<d(P)\leq X}\bigg(1-\frac{1}{2|P|}+O\bigg(\frac{1}{|P|^2}\bigg)\bigg)\\
&=\Big(1+O\big(q^{-X/2}/X\big)\Big) \frac{g^6}{45}\prod_{d(P)\leq X/2}\bigg(1-\frac{1}{|P|}\bigg)^6\prod_{X/2<d(P)\leq X}\bigg(1-\frac{1}{|P|}\bigg)^{9/2}\\
&=\frac{1}{720\sqrt{2}} \Big(\frac{g}{e^\gamma X/2}\Big)^6+O\big(g^6X^{-7}\big).
\end{align*}
The theorem follows by choosing any $0<\vartheta <(4-k)/2$.

\end{document}